\documentclass[12pt]{article}
\usepackage{amssymb,amsthm,amsmath,pb-diagram,color,enumerate}
\usepackage{eucal,mathrsfs}
\usepackage[all]{xy}
\usepackage{url}
\usepackage{fullpage}
\usepackage[OT2,T1]{fontenc}
\DeclareSymbolFont{cyrletters}{OT2}{wncyr}{m}{n}
\DeclareMathSymbol{\Sha}{\mathalpha}{cyrletters}{"58}

\newtheorem{thm}[equation]{Theorem}
\newtheorem{lemma}[equation]{Lemma}
\newtheorem{lem}[equation]{Lemma}
\newtheorem{prop}[equation]{Proposition}
\newtheorem{cor}[equation]{Corollary}
\theoremstyle{definition}
\newtheorem{remark}[equation]{Remark}
\newtheorem{rem}[equation]{Remark}

\newtheorem{defn}[equation]{Definition}
\newtheorem{example}[equation]{Example}
\newtheorem{ex}[equation]{Example}

\newcommand{\mbb}[1]{\mathbb #1}

\newcommand{\mc}[1]{\mathcal #1}
\newcommand{\ms}[1]{\mathscr #1}

\newcommand{\oper}[1]{\operatorname{#1}}
\newcommand{\wh}{\widehat}
\newcommand{\til}{\widetilde}
\newcommand{\GL}{\oper{GL}}
\newcommand{\PGL}{\oper{PGL}}
\newcommand{\F}{\mc F}

\newcommand{\Spec}{\oper{Spec}}

\newcommand{\Gal}{\oper{Gal}}
\newcommand{\Br}{\oper{Br}}
\newcommand{\W}{\oper{W}}
\newcommand{\per}{\oper{per}}
\newcommand{\ind}{\oper{ind}}

\newcommand{\wi}{\oper{i_W}}
\newcommand{\iso}{\to^{\!\!\!\!\!\!\!\sim\,}}

\newcommand{\lcm}{\oper{lcm}}
\newcommand{\cha}{\oper{char}}

\newcommand{\tors}{\oper{Tors}}

\newcommand{\Frac}{\oper{frac}}

\newcommand{\id}{\oper{id}}
\renewcommand{\O}{\oper{O}}
\newcommand{\SO}{\oper{SO}}
\newcommand{\SB}{\oper{SB}}
\newcommand{\op}{\oper{op}}

\newcommand{\pr}{\oper{pr}}
\def\<{\left<}
\def\>{\right>}

\newenvironment{compactenum}{
	\vspace{-.2cm}
	\begin{enumerate}[\ \ 1.]
	\setlength{\itemsep}{0cm}
	\setlength{\parskip}{0cm}

}{
	\end{enumerate}
	\vspace{-.2cm}
}

\newenvironment{shortenum}{
	\vspace{-.2cm}
	\begin{enumerate}[\ \ 1.]
	\setlength{\itemsep}{0.1cm}
	\setlength{\parskip}{0.1cm}

}{
	\end{enumerate}
	\vspace{-.2cm}
}

\newenvironment{compactitem}{
	\vspace{-.2cm}
	\begin{itemize}
	\setlength{\itemsep}{0cm}
	\setlength{\parskip}{0cm}
}{
	\end{itemize}
}

\newcounter{itemcounter}
\newenvironment{prooflist}{

	\mbox{}

	\setcounter{itemcounter}{1}
	\begin{list}{%
		\textit{Part \arabic{itemcounter}: }
	}{
		\usecounter{itemcounter}
		\setlength{\leftmargin}{0cm}
		\setlength{\labelsep}{0cm}
	}
}{
	\end{list}
}

\title{Refinements to patching \\ and applications to field invariants}
\author{David Harbater, Julia Hartmann, and Daniel Krashen}
\date{}
\numberwithin{equation}{section}
\begin{document}
\maketitle

%---------------------------------------------------------------------------
\begin{abstract}
We introduce a notion of refinements in the context of patching, in order to obtain new results about 
local-global principles and field invariants in the context of quadratic forms
and central simple algebras.  The fields we consider are finite extensions of the
fraction fields of two-dimensional complete domains that need not be local.  Our results in particular give the $u$-invariant and period-index bound for these fields, as consequences of more general abstract results.
\end{abstract}

%%%%%%%%%%%%%%%%%
%%%%%%%%%%%%%%%%%
\section{Introduction} \label{introduction}
%%%%%%%%%%%%%%%%%
%%%%%%%%%%%%%%%%%

In this manuscript we introduce the notion of refinements in the context of patching, and use this 
to obtain results about quadratic forms and central simple algebras over 
fraction fields of two-dimensional complete domains.  
These provide strengthenings and analogs of results in earlier papers.
Among our results here are local-global principles, which in the case of quadratic forms 
concern isotropy, the Witt group, the Witt index, and the $u$-invariant.  
In the case 
of central simple algebras they concern Brauer equivalence and the index.  
In addition, we obtain explicit results about the values of the $u$-invariant
and the period-index bounds for these fraction fields.

Classically, one relates the $u$-invariant and period-index bound for complete discretely valued fields to those of their residue fields.  Here, we consider the analogous situation of fraction fields of two-dimensional complete domains, which need not be local.  We focus on 
these two situations:
\begin{compactenum}
\renewcommand{\theenumi}{(\roman{enumi})}
\renewcommand{\labelenumi}{(\roman{enumi})}
\item 
the fraction field of a two-dimensional Noetherian complete local domain $R$ (e.g.~$k((x,t))$); 
\item 
a finite separable extension of the fraction field of the $t$-adic completion of $T[x]$, where $T$ is a complete discrete valuation ring with uniformizer $t$.
\end{compactenum}

In the context of central simple algebras, we obtain the following result.  (Our use of the term ``Brauer dimension'' is explained before Theorem~\ref{per-ind combination thm}.)

\begin{thm} \label{intro combined thm csa} 
In the above two situations, assume that the residue field $k$ of $R$ (resp.\ $T$) has 
Brauer dimension $d$ away from $p:=\cha(k)$.  Then $\ind(\alpha)$ divides $\per(\alpha)^{d+1}$
for all $\alpha \in \Br(E)$ whose period is not divisible by $p$.
\end{thm}

See Theorem~\ref{combined thm csa}, which also treats Brauer classes $\alpha \in \Br(E)$ of {\em arbitrary} period in the mixed characteristic case.  Using that, we obtain 
a local analog of \cite[Theorem~1]{PS:PIu}:

\begin{cor} \label{intro cor csa}
Let $L$ be the fraction field of $\mbb Z_p[[x]]$ or of the $p$-adic completion of  
$\mbb Z_p[x]$, and let $E$ be a finite extension of $L$.
Then $\ind(\alpha)$ divides $\per(\alpha)^2$ for all $\alpha \in \Br(E)$.
\end{cor}

Theorem~\ref{combined thm csa} also shows that 
the same conclusion holds if
instead $L$ is the fraction field of the $p$-adic completion of 
$\mbb Z_p^{\mathrm{ur}}[[x]]$ or $\mbb Z_p^{\mathrm{ur}}[x]$, with $\per(\alpha)=\ind(\alpha)$ if $\per(\alpha)$ is prime to $p$.  

For quadratic forms, we prove an analog of Theorem~\ref{intro combined thm csa}; see
Theorem~\ref{combined thm qf}.  This yields results about values of the $u$-invariant in 
mixed characteristic (Corollaries~\ref{mixed char uinv} and~\ref{char 2 res fld u-inv}), 
and also the following result in equicharacteristic:

\begin{thm} \label{intro thm qf}
Let $k$ be a field of characteristic unequal to two, and let $E$ be a finite separable extension of the fraction field of $k[x][[t]]$ or of $k[[x,t]]$.  Then
\begin{compactitem}
\item
$u(E)=4$ if $k$ is algebraically closed.
\item
$u(E)=8$ if $k$ is finite, or if $k=k_0((z))$ with $k_0$ algebraically closed.
\item
$u(E)=16$ if $k=k_0((z))$ with $k_0$ finite, or if $k=\mbb Q_p$ for some $p$ (which can equal $2$).
\item
$u(E)=32$ if $k=\mbb Q_p(z)$ or if $k=\mbb Q_p((z))$ for some $p$.
\end{compactitem}
\end{thm}

Note that the value of the $u$-invariant or the period-index bound for a given field does not in general give much information about the corresponding invariant for arbitrary finite separable extensions.  So the above results would not follow simply from knowing the values of these invariants for the fraction fields of rings of the form $T[x]$ or $T[[x]]$.

Previous results about the $u$-invariant and period-index bounds for related fields appeared in 
such papers as \cite{PS10}, \cite{HHK}, \cite{Leep}, \cite{Salt:DivAlg}, \cite{deJ}, and \cite{Lieb}.
To obtain our present results, we build on the patching framework that was 
used in our previous
manuscripts \cite{HH:FP}, \cite{HHK}, \cite{HHK:H1}, and \cite{HHK:Weier}.  

As in those papers, the patching framework also enables us to obtain
local-global principles.  
In particular, our Corollary~\ref{LGP isotropy local cor} proves a local-global principle for isotropy;
Corollary~\ref{uinv max} relates the $u$-invariants of fields to those of their completions; and Corollary~\ref{per-ind max} provides a local-global principle for the period-index bound of a field.
See also related results in \cite{CPS}, \cite{PS:PIu}, and \cite{HuLaurent}.  

The key new ingredient in this paper is a refinement principle for patching.  As
in the patching framework, we consider a projective normal curve over a complete discrete valuation ring, 
and we choose a finite partition of the closed fiber.
Criteria for patching and local-global principles are given in terms of intersection and factorization properties for a certain quadruple of rings arising from the partition.
By enlarging the given partition or modifying the model (e.g.\ by blowing up),
one can refine a quadruple, obtaining a 
new one with one part expanded.  The refinement principle that we state in this manuscript 
relates the intersection and factorization properties of a given quadruple to that of two other quadruples: the refined one, and the   
quadruple arising from the part that was expanded.
This principle is first stated in an abstract context (Proposition~\ref{subfactor}), and then used in a geometric context to establish ``patching on patches'' (Proposition~\ref{local patching}) and patching on exceptional divisors of a blow-up (Proposition~\ref{blowup patch}, answering a question of Yong Hu).  

The manuscript is organized as follows.  In Section~\ref{patch refine section}
we present general results about patching and local-global principles for quadruples of groups or rings, and then state our abstract refinement principle.  
In Section~\ref{patches} we turn to the geometric situation, 
generalizing the patching setup of \cite{HHK} in Section~\ref{setup} to allow more general
open subsets of the closed fiber, and then 
obtaining 
consequences of the refinement principle, in Sections~\ref{patching section}
and~\ref{fact quad subsect}.  We then turn to quadratic forms 
and central simple algebras in Section~\ref{applications}, first proving local-global results in an abstract context (Theorems~\ref{quad form abstract applic} and~\ref{Brauer}), and then specializing to the geometric situation to obtain results about numerical invariants, including those  
above.

We wish to thank Annette Maier, Yong Hu, and Jean-Louis Colliot-Th\'el\`ene for helpful discussions concerning results and ideas in this manuscript.
  
%%%%%%%%%%%%%%%%%
%%%%%%%%%%%%%%%%%
\section{Patching and refinement} \label{patch refine section}
%%%%%%%%%%%%%%%%%
%%%%%%%%%%%%%%%%%

This section proves a refinement principle (Proposition~\ref{subfactor}) that will afterwards permit us to obtain results about patching and local-global principles over certain function fields, once such results are known for related fields.  
In this section the presentation is in a more general framework.  
We begin with a discussion of patching and local-global principles, and the related conditions of factorization and intersection for quadruples.  
That discussion draws heavily on prior papers of the present authors.

%%%%%%%%%%%%%%%%%
\subsection{Diamonds of groups and rings}
%%%%%%%%%%%%%%%%%

Patching is a method that mimics constructions from complex geometry to obtain global objects from more local ones. In our algebraic version, we work with 
objects in a category $\mc C$ consisting of sets, possibly with additional structure (e.g.\ groups or rings), and we will consider quadruples of objects
$S_\bullet=(S, S_1,S_2, S_0)$ together with morphisms forming a commutative diagram
\[
\xymatrix @R=.1cm @C=.3cm {
& S_0 \\
S_1 \ar[ru]^{\beta_1} & & S_2 \ar[lu]_{\beta_2} \\
& S \ar[lu]^{\alpha_1} \ar[ru]_{\alpha_2}
} \eqno{(*)}
\]
As a motivating example, one can think of these as being the collection of functions or other objects on a global space $X$, on two subsets $U_1$ and $U_2$ that cover $X$, and on their intersection $U_0$, as indicated in the following commutative diagram:
\[
\xymatrix @R=.1cm @C=.3cm {
& U_0 \ar[ld]\ar[rd]\\
U_1 \ar[rd] & & U_2 \ar[ld] \\
& X
}
\]
For patching, the two key properties of a quadruple are factorization and intersection (see Theorem~\ref{patching prop thm}).  More precisely:

\begin{defn}  
\renewcommand{\theenumi}{\alph{enumi}}
\renewcommand{\labelenumi}{(\alph{enumi})}
\begin{enumerate}
\item
A {\em diamond} $S_\bullet$ in $\mc C$ consists of a commutative diagram $(*)$ as above
such that 
$(\alpha_1,\alpha_2):S \to S_1 \times_{S_0} S_2$ is a monomorphism.  
For short we will often write $S_\bullet=(S,S_1,S_2,S_0)$ as a quadruple if the maps $\alpha_i, \beta_i$ are understood.
\item
A diamond $S_\bullet$ as above has the
\textit{intersection property} if 
the map $(\alpha_1,\alpha_2):S \to S_1 \times_{S_0} S_2$ is an isomorphism.
It 
is {\em injective} if each of the maps $\alpha_i$, $\beta_i$ is injective.   
\item
Let $G_\bullet=(G, G_1,G_2, G_0)$ be a diamond of groups, together with maps $\alpha_i, \beta_i$. We say that $G_\bullet$ has
the \textit{factorization property} if $G_0 = \beta_1(G_1) \beta_2(G_2)$, i.e.\ every element of $G_0$ is of the form $\beta_1(g_1)\beta_2(g_2)$ with $g_i \in G_i$.
\end{enumerate}
\end{defn}

In the injective case, we often regard the maps $\alpha_i, \beta_i$ as inclusions; and to emphasize this, we write the diamond as $(S \le S_1, S_2 \le S_0)$.  With these identifications, the intersection property asserts that $S = S_1 \cap S_2$ in $S_0$; and the factorization property for diamonds of groups then asserts that each element of $G_0$ can be factored as $g_1g_2$ with $g_i \in G_i$.  By applying factorization to the inverse of each element of $G_0$, we obtain:

\begin{lem}\label{factor ordering}
A diamond of groups
$(G, G_1, G_2, G_0)$ has the factorization property (respectively the intersection
property) if and only if $(G, G_2, G_1, G_0)$ does.
\end{lem}

We will often consider injective diamonds of rings $F_\bullet=(F\le F_1,F_2\le F_0)$ with $F$ is a field and each $F_i$ a direct product of finitely many fields.  In particular, we have:

\begin{ex} \label{graph example}
Let $\Gamma$ be a bipartite connected (multi-)graph, with vertex set $\mc V = \mc V_1 \sqcup \mc V_2$ and edge set $\mc E$.  Suppose that we are given a $\Gamma$-field in the sense of \cite[Section~2.1.1]{HHK:Hi};
i.e.\ a field $F_v$ for each $v \in \mc V$ and a field $F_e$ for each $e \in \mc E$, together with an inclusion $F_v \hookrightarrow F_e$ whenever $v$ is a vertex of $e$.  These fields and inclusions define an inverse system of fields; and if the inverse limit is a field $F$ then this is called
a ``factorization inverse system'' over $F$ (\cite{HHK:H1}, Section~2), and the graph together with the associated fields is called a $\Gamma/F$-field (\cite{HHK:Hi}, Section~2.1.1).
In this situation, set
\(F_i=\prod_{v \in \mc V_i} F_v\) for $i=1,2$
and set \(F_0 = \prod_{e \in \mc E} F_e\).  Then the inclusions $F_v
\hookrightarrow F_e$ induce inclusions $F_i \hookrightarrow F_0$ for $i=1,2$;
and $F_\bullet=(F\leq F_1,F_2\leq F_0)$ is an injective diamond of the above form.
\end{ex}

In fact, every such diamond arises in this way:

\begin{prop} \label{diamonds from factorization systems}
Let $F_\bullet = (F \le F_1,F_2 \le F_0)$ be an injective diamond of rings having the intersection property, with $F$ a field and
each $F_i$ a finite product of fields. Then $F_\bullet$
is induced from a factorization inverse system over $F$ as in Example~\ref{graph example}.
\end{prop}

\begin{proof}
Write \(F_i = \prod_{\lambda \in \Lambda_i} F_\lambda\) with each
$F_\lambda$ a field, and let $\mc E = \Lambda_0, \mc V_1  = \Lambda_1, \mc V_2 =
\Lambda_2$. To give the structure of a graph $\Gamma$, we will
associate to every $e \in \mc E$ elements $v_i \in \mc V_i$ for $i =
1, 2$. To do this, choose $e' \in \mc E$ and for $i=1,2$ consider the composition
\(\prod_{v \in \mc V_i} F_v \to \prod_{e \in \mc E} F_e \to F_{e'}.\)
Since $F_{e'}$ is a field, the image of
this homomorphism is a domain. So the kernel must be a prime
and hence maximal ideal. Thus this composition
factors through a unique projection
\(\prod_{v \in \mc V_i} F_v \to F_{v'_i}.\)
The assignment of $e$ to $(v'_1, v'_2)$ gives a graph $\Gamma$, and the
above homomorphisms $F_{v_i'} \to F_{e'}$ give the structure of a $\Gamma$-field.
The inverse limit of the fields $F_v,F_e$ is the intersection $F_1 \cap F_2$,
which is equal to $F$ by the intersection property; and so these fields form a
factorization inverse system, which induces the diamond.
\end{proof}

%%%%%%%%%%%%%%%%%
\subsection{Patching and local-global principles}
%%%%%%%%%%%%%%%%%

To study patching and local-global principles in this framework, we will need 
to introduce the notion of diamonds of 
categories and tensor categories.

Let $\mc C_1, \mc C_2, \mc C_0$ be categories, and suppose we
are given functors $G_i: \mc C_i \to \mc C_0$. The (2-)fiber product
$\mc C_1 \times_{\mc C_0} \mc C_2$ is defined to be the category whose
objects are triples $(C_1, C_2, \phi)$ where $C_i \in \mc C_i$ and
$\phi: G_1(C_1) \to G_2(C_2)$ is an isomorphism. A morphism $(C_1,
C_2, \phi) \to (D_1, D_2, \psi)$ is defined to be a pair or morphisms
$f_i : C_i \to D_i$ such that we have a commutative square
\[
\xymatrix @R=.5cm @C=.3cm {
G_1(C_1) \ar[d]_{f_1} \ar[r]^{\phi} & G_2(C_2) \ar[d]^{f_2} \\
G_1(D_1) \ar[r]_\psi & G_2(D_2)
}
\]

\begin{defn}[Patching Problems]\label{patching prob def}
A {\em diamond} of (tensor) categories is a 
diagram
\[
	\xymatrix @R=.3cm @C=.3cm {
		& \mc C_0 \\
		\mc C_1 \ar[ru]^{G_1} & & \mc C_2 \ar[lu]_{G_2} \\
		& \mc C \ar[lu]^{F_1} \ar[ru]_{F_2}  
}
\]
 of (tensor) categories and functors, together with a natural isomorphism of functors
\(\alpha: G_1F_1 \to G_2F_2\), such that the functor
\(\Phi: \mc C \to \mc C_1 \times_{\mc C_0} \mc C_2\),
given by $\Phi(c) = (F_1(c), F_2(c), \alpha(c))$, is
essentially injective.  In this situation:
\begin{compactenum}
\renewcommand{\theenumi}{\alph{enumi}}
\renewcommand{\labelenumi}{(\alph{enumi})}
\item
A \textit{patching problem} is an object $C_\bullet$ in the fiber
product category $\mc C_1 \times_{\mc C_0} \mc C_2$.
\item
A \textit{solution} to a patching problem 
$C_\bullet$ is an object $C \in \mc C$ such that $\Phi(C) \cong
C_\bullet$.
\item
\textit{Patching holds} for the diamond if 
$\Phi$ is an equivalence of categories.
\end{compactenum}
\end{defn}

\begin{ex}[Patching for torsors]
If $R_\bullet=(R, R_1,R_2, R_0)$ is a diamond of rings, and $G$ is an algebraic group
over $R$, we obtain a diamond of categories of torsors
$\tors(G_{R_\bullet})$. In this case, we refer to
patching problems (solutions, etc.) for $\tors(G_{R_\bullet})$ as patching
problems (solutions, etc.) for $G$-torsors.
\end{ex}

\begin{ex}[Patching for free modules] \label{other categories}
Similarly, if $R_\bullet=(R, R_1,R_2, R_0)$ is a diamond of rings, we obtain a diamond $\F({R_\bullet})$
of tensor categories of free modules of finite rank. If patching
holds for $\F(R_\bullet)$, it follows that patching holds for other
categories of structures which may be defined in terms of the category
of vector spaces and its tensor structure (e.g.\ central simple algebras).  
The proof is as in \cite{HH:FP},
Theorems~7.1 and 7.5.
\end{ex}

\begin{thm}\label{patching prop thm}
Let $R_\bullet=(R, R_1,R_2, R_0)$ be a diamond of rings.
\renewcommand{\theenumi}{\alph{enumi}}
\renewcommand{\labelenumi}{(\alph{enumi})}
\begin{enumerate}
\item \label{patching and matrix factorization}
The following conditions are equivalent:
\begin{enumerate}[\ \ (i)]
\item \label{matrix factorization cond}
For every $n \geq 1$, the diamond of groups $\operatorname{GL}_n(R_\bullet)$ satisfies the intersection and
factorization properties.
\item \label{patching for free modules cond}
Patching holds for free modules; that is, for the diamond of
categories
$\F(R_\bullet)$.
\end{enumerate}
\item \label{inverse to equivalence}
Under these conditions, the inverse to the equivalence
$\Phi:\F(R) \to \F(R_1) \times_{\F(R_0)} \F(R_2)$ is given by taking the intersection of free modules.
\item \label{patching torsors}
Suppose that $R$ is a field, that each $R_i$ is a finite product of fields, and the diamond is injective.  Then the above two conditions are also equivalent to:
\begin{enumerate}[\!\!\!\!\!\!\!\!(iii)]
\item \label{patching for torsors cond}
For every linear algebraic group $G$ over $R$, patching holds for $G$-torsors,
i.e.\ for $\tors(G_{R_\bullet})$.
\end{enumerate}
\end{enumerate}
\end{thm}

\begin{proof}
First observe that if $\Phi$ is an equivalence of categories, then necessarily
$R$ is the fiber product $R_1 \times_{R_0} R_2$.  Namely, for any $c \in
R_1 \times_{R_0} R_2$, consider the endomorphism of the rank one object
$(R_1,R_2,\id)$ 
given by multiplication by $c$.  Since $\Phi$ is an
equivalence, this morphism is induced by an endomorphism of the free rank one
$R$-module $R$; i.e.\ by multiplication by some element of $R$, which is
necessarily equal to $c$.  Thus $R_1 \times_{R_0} R_2 = R$.
The first two parts of the assertion now follow from \cite[Proposition~2.1]{Ha:CAPS}, which
says that if $R = R_1 \times_{R_0} R_2$, then $\Phi$ is an equivalence of categories
if and only if factorization holds; and moreover that in this case the inverse
of $\Phi$ is given by taking the fiber product of objects.  
The third part follows from \cite[Theorem~2.1.4]{HHK:Hi}, which applies
here by Proposition~\ref{diamonds from factorization systems} above.
\end{proof}

\begin{defn} \label{patching prop def}
\textit{Patching holds} for the diamond of rings $R_\bullet$ if
either of the equivalent conditions of Theorem~\ref{patching prop thm}(\ref{patching and matrix factorization})
holds.
\end{defn}

\begin{lem} \label{corestriction}
Let $R_\bullet=(R, R_1,R_2, R_0)$ be a diamond of rings.  Let $R
\subseteq S$ be a finite extension of rings such that $S$ is a free
$R$-module.   Set $S_i = S \otimes_R R_i$.
\begin{enumerate}
\renewcommand{\theenumi}{\alph{enumi}}
\renewcommand{\labelenumi}{(\alph{enumi})}
\item \label{corestriction int}
If the intersection property holds for
$R_\bullet$, then it also holds for
$S_\bullet$.
\item \label{corestriction fact}
If patching holds for $R_\bullet$, then it also holds for $S_\bullet$.
\end{enumerate}
\end{lem}
\begin{proof}
For the first part, suppose that $R_\bullet$ has the intersection
property.  We thus have a left exact
sequence
\(0 \to R \to R_1 \times R_2 \to R_0\) of $R$-modules,
where the map on the right is given by 
subtracting the image under $R_2 \to R_0$ from the image under $R_1 \to R_0$.
Since $S$ is free over $R$, the sequence
\(0 \to S \to S_1 \times S_2 \to S_0\)
is also left exact. Consequently $S_\bullet$ also has the intersection
property.

For the second statement, let $(s_j)$ be a basis for the free $R$-module $S$ and
suppose that $s_j s_k = \sum a_{j,k}^\ell s_\ell$. Then the category
of finitely generated free $S_i$-modules is equivalent to the category whose objects are
finitely generated free $R_i$-modules together with $R_i$-endomorphisms $\til
s_j$ (corresponding to multiplication by $s_j$) such that
$\til s_j \til s_k = \sum a_{j,k}^\ell \til s_\ell$, and where the
morphisms in the category are required to commute with each $\til s_j$. In particular,
since patching holds for the diamond of categories
$\F(R_\bullet)$, it follows that patching also holds for
$\F(S_\bullet)$.  That is, patching holds for $S_\bullet$.
\end{proof}

\begin{defn} \label{local-global def}
Let $R_\bullet=(R, R_1,R_2, R_0)$ be a diamond of rings.  Let $\mc V$
be a class of $R$-varieties.  We say that the {\em local-global principle
  holds} for $\mc V$ with respect to $R_\bullet$ if for each $V \in \mc V$,
the condition $V(R_i) \neq \emptyset$ for $i = 1, 2$ implies that $V(R)\neq \emptyset$.
\end{defn}

This definition applies in particular to the key case considered above, 
where we are given an injective diamond 
$(F \le F_1,F_2 \le F_0)$, with $F$ is a field and each $F_i$
a finite product of fields $\prod_j F_{ij}$, and where we take 
$\mc V$ to be the class of $G$-torsors over $F$, for some
linear algebraic group $G$ over $F$.  The set of isomorphism classes of
$G$-torsors over $F$ is in natural bijection with the pointed Galois
cohomology set $H^1(F, G)$, and we write
$H^1(F_i,G)$ for $\prod_j H^1(F_{ij},G)$.  The local-global principle
holds for (the class of) $G$-torsors if and only if the natural map
\(\phi: H^1(F, G) \to H^1(F_1, G) \times H^1(F_2, G)\)
has trivial kernel.

Given a diamond $R_\bullet=(R, R_1,R_2, R_0)$ of rings, and a 
linear algebraic group $G$ over $R$, there is an associated
diamond of groups
$G(R_\bullet)=(G(R), G(F_1),G(F_2), G(F_0))$ of rational points.  By embedding 
$G$ in $\GL_{n,R} \subset \mbb A^N_R$ and considering coordinates, 
we immediately obtain the following lemma, which 
was implicitly used in \cite{HHK}, Section~3.

\begin{lem} \label{intersection for groups}
Suppose that $R_\bullet$ is a diamond of rings with the intersection
property and that $G$ is a linear algebraic group over $R$. Then
$G(R_\bullet)$ also has the intersection property.
\end{lem}

\begin{thm} \label{factor local-global}
Suppose that $F_\bullet=(F\leq F_1,F_2\leq F_0)$ is an injective diamond of rings with
$F$ a field and each $F_i$ a finite direct product of fields. Assume moreover
that patching holds for $F_\bullet$.
Then the following statements are equivalent
for a linear algebraic group $G$ over $F$:
\begin{enumerate}
\renewcommand{\theenumi}{(\roman{enumi})}
\renewcommand{\labelenumi}{(\roman{enumi})}
\item $G(F_\bullet)$ satisfies factorization. \label{factorization for G}
\item The local-global principle holds for $G$-torsors with respect to $F_\bullet$. \label{local-global for G}
\item The local-global principle holds, with respect to $F_\bullet$, for the
  class of $F$-varieties $V$ equipped with a $G$-action such that $G(F_0)$
  acts transitively on $V(F_0)$. \label{local-global for homogeneous spaces}
\end{enumerate}
\end{thm}

\begin{proof}
By Proposition~\ref{diamonds from factorization systems}, the diamond $F_\bullet$ arises from a factorization inverse system; and so \cite[Theorem~2.4]{HHK:H1} applies.  Thus there is the following exact sequence of pointed sets:
\[H^0(F_1,G) \times H^0(F_2,G) \to H^0(F_0,G) \to H^1(F,G) \to H^1(F_1,G)
\times H^1(F_2,G).\]
Condition~\ref{factorization for G} asserts surjectivity of the first arrow,
which is equivalent to the third arrow having trivial kernel.  That latter
property is the same as condition~\ref{local-global for G}, as discussed
above.  Moreover the triviality of that kernel implies
condition~\ref{local-global for homogeneous spaces} by
\cite[Corollary~2.8]{HHK:H1}.  Finally, condition~\ref{local-global for
homogeneous spaces} trivially implies  condition~\ref{local-global for
G}, since $G$-torsors over $F$ satisfy the transitivity hypothesis of
condition~\ref{local-global for homogeneous spaces}.
\end{proof}

%%%%%%%%%%%%%%%%%
\subsection{A refinement principle} \label{refinement subsec}
%%%%%%%%%%%%%%%%%
In studying factorization and intersection for a diamond $G_\bullet=(G,
E,H,J)$ of groups, it will prove useful to consider the situation
when $H$ is itself the base for
another diamond of groups $H_\bullet=(H,H_1,H_2,H_0)$. 
We would then like to combine the two
diamonds into one new diamond, thereby refining the original
situation. The natural question is to what extent such refinements preserve
factorization and intersection.

Below, given maps $f:A \to B$, $g:A \to B'$, and $f':A' \to B'$, we write
$(f,g):A \to B \times B'$ for the map $a \mapsto (f(a),g(a))$, and write
$f \times f':A \times A' \to B \times B'$ for 
$(a,a') \mapsto (f(a),f'(a'))$.

\begin{prop}[Refinement Principle] \label{subfactor}
Suppose we are given a commutative diagram of groups and
homomorphisms
\[\ms R = \fbox{\xymatrix @R=.5cm @C=.7cm {
& J & & H_0 \\
E\ar[ru]^{\beta_2} & & H_1 \ar[lu]_{\beta_1}
\ar[ru]^{\gamma_1} & & H_2 \ar[lu]_{\gamma_2} \\
& & & H \ar[lu]^{\varepsilon_1} \ar[ru]_{\varepsilon_2} \\
& & G \ar[lluu]^{\delta_2}\!\! \ar[ru]_{\delta_1}
 }}\]
such that the following diagrams are diamonds of groups:
\vspace{.3cm}

\centerline{
\(
G_\bullet = \fbox{
	\xymatrix @R=.5cm @C=.3cm {
		& J \\
		E \ar[ru]^{\beta_2} & & H \ar[lu]_{\beta_1\varepsilon_1} \\
		& G \ar[lu]^{\delta_2} \ar[ru]_{\delta_1}
	}
} 
\ \ \ G'_\bullet = \fbox{
\xymatrix @R=.5cm @C=.3cm {
& J \times H_0 \\
H_1 \ar[ru]^-{(\beta_1, \gamma_1)} & &
\!\!\!\!\! E \times H_2 \, \ar[lu]_-{\beta_2\times \gamma_2} \\
& G \ar[ru]_-{(\delta_2, \varepsilon_2\delta_1)} \ar[lu]^-{\varepsilon_1\delta_1}
}} \)
\(
\ \ H_\bullet = \fbox{
	\xymatrix @R=.5cm @C=.3cm {
		& H_0 \\
		H_1 \ar[ru]^{\gamma_1} & & H_2 \ar[lu]_{\gamma_2} \\
		& H \ar[lu]^{\varepsilon_1} \ar[ru]_{\varepsilon_2}
	}
}
\)
}

\noindent Then
\begin{shortenum}
\item \label{subfactor 1}
If $G'_\bullet$ has the factorization property then so does $H_\bullet$.
\item \label{subfactor 2}
If $G'_\bullet$ has the factorization property and $H_\bullet$ has
the intersection property then $G_\bullet$ has the factorization
property.
\item \label{subfactor 3}
If $G'_\bullet$ has the intersection property then so
does $G_\bullet$.
\item \label{subfactor 4}
If $G'_\bullet$ has the intersection property and
$G_\bullet$ has the factorization property then $H_\bullet$ has
the intersection property.
\item \label{subfactor 5}
If $G_\bullet$ and $H_\bullet$ have the intersection property
then so does $G'_\bullet$.
\item \label{subfactor 6}
If $G_\bullet$ and $H_\bullet$ have the factorization property
then so does $G'_\bullet$.
\end{shortenum}
\end{prop}

\begin{proof}
\begin{prooflist}
\!\!\!\!\item
Assume that
$G_\bullet'$ has the factorization property. Let $h_0 \in H_0$.
By hypothesis,
we may write
$(1, h_0) \in J \times H_0$ as
$(\beta_1,\gamma_1)(h_1)\cdot(\beta_2(e),\gamma_2(h_2))$ for some elements
$h_1 \in H_1$ and $(e, h_2) \in E \times H_2$.  In particular, $\gamma_1(h_1)\gamma_2(h_2) = h_0$,
hence $H_\bullet$ has the factorization property.

\item
We now additionally assume that $H_\bullet$ has the
intersection property. To see that $G_\bullet$ has the
factorization property, suppose that $j \in J$ and consider $(j,1) \in J \times H_0$.
Using factorization for $G'_\bullet$, we may
find $h_1 \in H_1$ and $(e, h_2) \in E\times H_2$  such that $j =
\beta_1(h_1)\beta_2(e)$ and $1 = \gamma_1(h_1)\gamma_2(h_2)$. Since by the latter equality, $h_1$ and $h_2^{-1}$ have the same
image in $H_0$,
the intersection hypothesis for
$H_\bullet$ implies that there exists
$h \in H$ with
$\varepsilon_1(h) = h_1$ and $\varepsilon_2(h) = h_2^{-1}$.  Hence
$j = \beta_1\varepsilon_1(h) \beta_2(e)$, which proves factorization for 
$G_\bullet$ (see Lemma~\ref{factor ordering}).

\item
Suppose that $G_\bullet'$ has the
intersection property, and let $(e, h) \in
E \times H$ satisfy
$\beta_2(e)=\beta_1\varepsilon_1(h)$. We wish to show that
$e$ and $h$ are the images of a single
element $g \in G$. To see this, we first note that the
elements $\varepsilon_1(h) \in H_1$ and
$(e, \varepsilon_2(h)) \in E \times H_2$
have the same image in $J\times H_0$. Since $G'_\bullet$ has the
intersection property, we may find $g
\in G$ so that
$\varepsilon_1\delta_1(g) = \varepsilon_1(h)$, $\delta_2(g) = e$,
and $\varepsilon_2\delta_1(g)=\varepsilon_2(h)$.  But
$(\varepsilon_1,\varepsilon_2)$ is injective, since $H_\bullet$ is assumed to be
a diamond.  It follows that $\delta_1(g) =
h$, and $g$ is as desired.

\item
Suppose that $G'_\bullet$ has the intersection property
and $G_\bullet$ has the factorization property. 
Assume that $h_0 =
\gamma_1(h_1) = \gamma_2(h_2)$ with $h_i \in H_i$. We would like to show that
$h_1,h_2$ are the images in $H_1,H_2$ of some element of $H$. Using the
factorization property for $G_\bullet$, we may write
$\beta_1(h_1) \in J$ as
$ \beta_2(e)\cdot \beta_1\varepsilon_1(h)$ with
$e \in E$, $h \in H$.
Let $h_1' = h_1 \varepsilon_1(h)^{-1} \in H_1$ and
$h_2' = h_2 \varepsilon_2(h)^{-1} \in H_2$.  Thus $h_1',h_2'$ have the same image in $H_0$
under $\gamma_1, \gamma_2$ respectively,
viz.\ the element $h_0' := h_0 \,\gamma_1\varepsilon_1(h)^{-1} = h_0 \,\gamma_2\varepsilon_2(h)^{-1}$.
Moreover $\beta_1(h_1') = \beta_2(e)$.
Consider the images of $h_1' \in H_1$ and of $(e,h_2')\in E\times H_2$ in
$J\times H_0$. These are
$(\beta_1, \gamma_1)(h_1')$ and $(\beta_2(e), \gamma_2(h_2'))$, which by the
previous considerations are equal.

Since
$G'_\bullet$ has the intersection property, there exists an element $g \in G$
for which $\varepsilon_1\delta_1(g) = h_1'$
and $\varepsilon_2\delta_1(g) = h_2'$.
So $\delta_1(g) \in H$ maps to $h_1' \in H_1$ and $h_2' \in H_2$
under $\varepsilon_1,\varepsilon_2$.
Thus $h_1 \in H_1$ and $h_2 \in H_2$
are the images of the common element $\delta_1(g) h \in H$.

\item
Suppose $h_1 \in H_1$ and $(e,h_2) \in E \times H_2$
satisfy
$(\beta_1(h_1),\gamma_1(h_1))
= (\beta_2(e),\gamma_2(h_2))$.
The intersection property for $H_\bullet$ yields an element $h \in H$ such that
$\varepsilon_i(h)=h_i$ for $i=1,2$. The intersection property for $G_\bullet$
then yields an element $g \in G$ such that  $\delta_2(g)=e$ and
$\delta_1(g)=h$.
Thus $\varepsilon_1\delta_1(g)=\varepsilon_1(h)=h_1$ and
$(\delta_2,\varepsilon_2\delta_1)(g)=(e,h_2)$; i.e.\
$h_1$ and $(e,h_2)$ are the images of the common element $g \in G$.

\item
Let $(j,h_0) \in  J \times H_0$.  By factorization
for $H_\bullet$, there exist $h_i \in H_i$ for $i=1,2$, such that
$h_0=\gamma_1(h_1)\gamma_2(h_2)$.
By factorization for $G_\bullet$ and Lemma~\ref{factor ordering},
there exist
$h \in H$ and $e \in E$ such that
$\beta_1(h_1)^{-1}j \in J$ equals
$\beta_1\varepsilon_1(h)\cdot \beta_2(e)$;
i.e.\
$j = \beta_1(h_1\varepsilon_1(h))\cdot \beta_2(e) \in J$.
Moreover $h_0 = \gamma_1(h_1)\gamma_2(h_2) =
\gamma_1(h_1)\gamma_1\varepsilon_1(h)\gamma_2\varepsilon_2(h)^{-1}\gamma_2(h_2) =
\gamma_1(h_1\varepsilon_1(h))\gamma_2(\varepsilon_2(h)^{-1}h_2)$.
Thus the elements
$h_1\varepsilon_1(h) \in H_1$ and $(e,\varepsilon_2(h)^{-1}h_2) \in E \times H_2$
provide a factorization of $(j,h_0) \in J \times H_0$.
\end{prooflist}
\vspace{-8.5mm}
\end{proof}

The following result will be useful in conjunction with the above proposition.

\begin{lem} \label{enlarging diamonds}
\renewcommand{\theenumi}{\alph{enumi}}
\renewcommand{\labelenumi}{(\alph{enumi})}
\begin{enumerate}
\item \label{product diamond}
Let $H^{(j)}_\bullet = (H^{(j)},H^{(j)}_1,H^{(j)}_2,H^{(j)}_0)$ be a diamond of groups for each $j$.
Write $H = \prod H^{(j)}$ and $H_i = \prod H^{(j)}_i$ for $i=0,1,2$, and
let $H_\bullet = (H,H_1,H_2,H_0)$, together with the products of the maps defining the diamonds $H^{(j)}_\bullet$.  Then $H_\bullet$ is a diamond, and it satisfies intersection (resp.\ factorization) if and only if each $H^{(j)}_\bullet$ does.
\item \label{extend diamond}
Let $\til H_\bullet = (\til H,\til H_1,\til H_2,\til H_0)$ be a diamond of groups, with associated maps 
$\alpha_i:\til H \to \til H_i$ and $\beta_i:\til H_i \to \til H_0$ for $i=1,2$.  
Let $A$ be a group and write $H = \til H \times A$, $H_1 = \til H_1 \times A$, and 
$H_i=\til H_i$ for $i=0,2$.
Then $H_\bullet := (H,H_1,H_2,H_0)$ is a diamond with respect to the maps 
$\alpha_1\times\id, \alpha_2 \circ \pr_1, \beta_1 \circ \pr_1, \beta_2$, where 
$\pr_1$ is the first projection map.  Moreover $H_\bullet$ satisfies 
intersection (resp.\ factorization) if and only if $\til H_\bullet$ does.
\end{enumerate}
\end{lem}

Here part~(\ref{product diamond}) is immediate, and part~(\ref{extend diamond})
then follows by applying part~(\ref{product diamond}) to the two diamonds 
$\til H_\bullet$ and $(A,A,1,1)$.

%%%%%%%%%%%%%%%%%%%
%%%%%%%%%%%%%%%%%
\section{Patches and their fields} \label{patches}
%%%%%%%%%%%%%%%%%%%
%%%%%%%%%%%%%%%%%

In this section, we apply the refinement principle, Proposition~\ref{subfactor}, to fields arising from curves over complete discretely valued fields.  
In that situation, it was shown in \cite{HHK} that patching holds for diamonds arising from a partition of the closed fiber of a normal projective model of the curve into finitely many closed points and irreducible open sets.  
Here we show the same holds for more general partitions of the closed fiber (Proposition~\ref{big patch ok}); for partitions of a connected open subset of the closed fiber (Proposition~\ref{local patching}); and for partitions of the exceptional divisor of a blow-up (Proposition~\ref{blowup patch}).  
The second of these can be regarded as an assertion about 
``patching on patches.''  Related results for factorization with respect to  algebraic groups, which will be used in Section~\ref{applications}, appear 
in Section~\ref{fact quad subsect}.

For the sake of the applications in Section~\ref{applications}, we will need to consider reducible open sets in our partitions, in order to be able to treat branched covers that have reducible closed fibers over a given open subset of the base. 
This will require us first to generalize somewhat the framework that was considered in \cite{HH:FP} and \cite{HHK}, where the open sets had been assumed to be irreducible.  (See also \cite{Cuong}, where a similar generalization is considered.)   

%%%%%%%%%%%%%%%%%
\subsection{Setup} \label{setup}
%%%%%%%%%%%%%%%%%

Consider a complete discrete valuation ring $T$ with uniformizer $t$,
residue field $k$, and fraction field $K$.  Let $F$ be a one-variable
function field over $K$, and let $\wh X$ be a \textit{normal model} of
$F$, i.e.\ a normal connected projective $T$-curve with function field
$F$. Let $X$ be the reduced closed fiber of $\wh X$.

The following definition generalizes the notation in \cite[Section~6]{HH:FP}
and \cite{HHK}, where the open sets of $X$ were required to be irreducible.
\begin{defn} \label{patch ring def}
For a point $P \in X$ we let $\mc O_{\wh X, P}$ be its local ring, consisting of
the elements of $F$ that are regular at $P$. For a nonempty strict open
subset $W
\subset X$, we define
\(\displaystyle R_W = \bigcap_{P \in W} \mc O_{\wh X, P},\)
and we let $\wh R_W$ be the $t$-adic completion of $R_W$. If $\wh R_W$ is a domain
we also define $F_W$ to be its fraction field.
\end{defn}

Thus $R_W$ is the subring of $F$ consisting of the rational functions on $\wh X$ that are regular at each point of $W$.  The above definition agrees with those in \cite{HH:FP} and \cite{HHK}, which considered $R_W$ and $\wh R_W$ only when $W$ meets just one irreducible component of $X$.
For $W$ an affine open subset of $X$, we view $\Spec(\wh R_W)$ as a ``thickening'' of $W$, just as $\wh
X$ is a ``thickening'' of its closed fiber $X$.  See
Lemma~\ref{lemma_tilde}(\ref{closed_fiber_points-affine}) below.
By convention, we also write $F_X:=F$. If we have a second curve $\wh X'$ with
function field $F'$, we will write $R_W'$, $\wh R_W'$, and $F_W'$ for the analogously
defined 
rings (where $W$ is now a nonempty strict open subset of the closed
fiber $X'$ of $\wh X'$).

\begin{rem} \label{RW remark}
\renewcommand{\theenumi}{\alph{enumi}}
\renewcommand{\labelenumi}{(\alph{enumi})}
\begin{enumerate}
\item \label{normal patch rings remark}
For $W$ a non-empty open subset of $X$ as above,
and for any point $P \in W$, the completion of $\wh R_W$
at the ideal defined by $P \in \wh X$ is the complete local ring
of $\wh X$ at $P$, which is a normal domain.
Moreover, since $R_W$ is normal and reduced (i.e.\ has no nilpotents),
the same conditions hold for its
completion $\wh R_W$ (by \cite{Mat}, 33.I, 34.A); this uses that $T$
and hence $R_W$ is excellent (by \cite{Mat}, 34.B, 34.A).
\item \label{contraction remark}
When studying rings of the form $\wh R_W$ on normal models of $F$, it suffices to restrict attention to affine open sets $W$.
This is for the following reason:
Given a collection of some but not all irreducible components of $X$,
by \cite[Sect.~6.7, Proposition~4]{BLR:Neron} there is
an associated contraction $\pi:\wh X \to \wh Y$.  This is a projective
birational $T$-morphism $\pi$ to a normal model $\wh Y$ of $F$ over $K$
that is an isomorphism away from these components, and which sends
each of these components to a point.  By the description of $\wh Y$
given in \cite[Sect.~6.7, Theorem~1]{BLR:Neron},
if $U$ is an open set strictly contained in the closed fiber $Y$ of $\wh Y$,
and $W = \pi^{-1}(U)$, then the $T$-morphism $\pi$ induces a $T$-algebra
isomorphism between $R_W$ (taken on $\wh X$) and $R_U$
(taken on $\wh Y$).
This in turn induces an isomorphism between the $t$-adic completions $\wh R_W$
and $\wh R_U$.
In particular, let $W$ be any connected open set strictly contained in $X$, and let $\pi$ be chosen to contract precisely those components of $X$ that are contained in $W$.  Then $\wh R_W = \wh R_U$ for $U=\pi(W)$, and moreover $U$ is affine.
\end{enumerate}
\end{rem}

With $W \subset X$ as above, the reduced closed fiber of $\Spec(R_W)$
is $\Spec(R_W/J)$, where $J$ is the Jacobson radical of $R_W$
(i.e.\ the radical of the ideal $tR_W$).  The corresponding statement
holds for $\Spec(\wh R_W)$.

\begin{lemma} \label{lemma_tilde}
In the above situation, let $W$ be a non-empty open subset of $X$.
\renewcommand{\theenumi}{\alph{enumi}}
\renewcommand{\labelenumi}{(\alph{enumi})}
\begin{enumerate}
\item \label{closed_fiber_points-affine}
If $W$ is an affine open subset of $X$, then the reduced closed fibers of $\Spec(R_W)$ and $\Spec(\wh R_W)$ are each isomorphic to $W$.
\item \label{closed_fiber_points-general}
More generally, if $W \ne X$, then the reduced closed fibers of $\Spec(R_W)$ and $\Spec(\wh R_W)$ are each isomorphic to $\pi(W)$, where $\pi$ is the contraction of $\wh X$ with respect to the irreducible components of $X$ that are contained in $W$.
\end{enumerate}
\end{lemma}

\begin{proof}
(\ref{closed_fiber_points-affine})
The assertion for $R_W$ is clear, and it then follows for $\wh R_W$ because
$R_W$ and $\wh R_W$ have the same reduction modulo $(t)$.

(\ref{closed_fiber_points-general})
Note that $\pi(\tilde W)$ is open because the only components of $W$ that are contracted by $\pi$ are those contained in $W$.  Also
$R_W = R_{\pi(W)}$, as observed in Remark~\ref{RW remark}(\ref{contraction remark}). So the assertion follows.
\end{proof}

\begin{prop} \label{connec_conds}
Let $W$ be a nonempty proper open subset of $X$, the closed fiber of $\wh X$.  Let $W_1,\dots,W_n$ be the connected components of $W$.
\renewcommand{\theenumi}{\alph{enumi}}
\renewcommand{\labelenumi}{(\alph{enumi})}
\begin{enumerate}
\item \label{patch ring connec}
The ring $\wh R_W$ is a domain (and so $F_W$ is defined) if and only if
$W$ is connected.
\item \label{patch ring structure}
The ring $\wh R_W$ is isomorphic to $\prod_{i=1}^n \wh R_{W_i}$.
\item \label{extension structure}
Let $F'$ be a finite extension of $F$, and let $\wh X'$ be the normalization of $\wh X$ in $F'$,
with associated morphism $\pi:\wh X' \to \wh X$ and closed fiber $X'$.
Let $W' = \pi^{-1}(W) \subset X'$. 
Then $\wh R_W \otimes_{R_W} R_{W'}'$
is isomorphic to $\wh R_{W'}' = \prod_{j=1}^n \wh R_{W'_j}'$,
where $W'_1,\dots,W'_n$ are the connected components of $W'$.
If $W$  is connected, 
$F_W \otimes_F F'$ is isomorphic to $\prod_{j=1}^n F'_{W'_j}$.
\end{enumerate}
\end{prop}

\begin{proof}
We begin with part~(\ref{patch ring structure}).
By Lemma~\ref{lemma_tilde}(\ref{closed_fiber_points-general}),
the reduced closed fiber of $\Spec(R_W)$ is the disjoint union of
the reduced closed fibers of $\Spec(R_{W_i})$, for $i=1,\dots,r$.
Thus the ideal $J \subset R_W$ defining the former closed fiber is the
product of the relatively prime ideals $J_i \subset R_W$ defining the latter;
and more generally $J^n=\prod_{i=1}^r J_i^n$ for all $n$.
The Chinese Remainder Theorem implies that
$R_W/J^n =  \prod_{i=1}^r R_W/J_i^n
= \prod_{i=1}^r R_{W_i}/J_i^n R_{W_i}$ for all $n$.
Since $J$ is the radical of the ideal $(t) \subset R_W$, and similarly for $J_i$ and $R_{W_i}$,
the asserted isomorphism follows by passing to the inverse limit.

The forward direction of part~(\ref{patch ring connec}) is now immediate.
For the reverse implication of~(\ref{patch ring connec}), note that
the condition that $\wh R_W$ is a domain is equivalent to its spectrum being reduced and irreducible.  It was observed in
Remark~\ref{RW remark}(\ref{normal patch rings remark})
that $\wh R_W$ (or equivalently its spectrum) is reduced and normal.  Moreover
a normal scheme is irreducible if and only if it is connected.
Thus it suffices to show that $\Spec(\wh R_W)$ is connected.

So suppose that $\Spec(\wh R_W)$ is the disjoint union of two Zariski open subsets $Y_1$ and $Y_2$.  We wish to show that $Y_1$ or $Y_2$ is empty.
First note that $\Spec(\wh R_W/(t))$ is the disjoint union of the two Zariski open subsets $\bar Y_1$ and $\bar Y_2$, the restrictions of $Y_1$ and $Y_2$ to $\Spec(\wh R_W/(t))$.
Also, $\Spec(\wh R_W/(t))$ is connected since $W$ and hence $\pi(W)$ is, using Lemma~\ref{lemma_tilde}(\ref{closed_fiber_points-general}).  So either $\bar Y_1$ or $\bar Y_2$ is empty.
But any maximal ideal of $\wh R_W$ contains $t$, and so any closed point of $\Spec(\wh R_W)$ (including any closed point of $Y_i$) lies on $\Spec(\wh R_W/(t))$.  Hence if $Y_i$ is non-empty, then so is $\bar Y_i$.  Thus either $Y_1$ or $Y_2$ is empty, concluding the proof of~(\ref{patch ring connec}).

Finally, we prove part~(\ref{extension structure}).
The map $\wh R_W \otimes_{R_W} R_{W'}' \to \wh R_{W'}'$ is an isomorphism
by \cite[III, \S 3.4, Theorem 3(ii)]{Bo:CA}.  For the second assertion, where $W$ is connected,
write $S = \wh R_W^\times$.  Then the localization
$S^{-1}\wh R'_{W'_j} = F_W \otimes_{\wh R_W} \wh R'_{W'_j}$
is a domain that is a finite extension of $F_W$.
Thus it is a field, and is equal to its fraction field $F'_{W'_j}$.
So $F_W \otimes_F F' = F_W \otimes_{R_W} R_{W'}'
= F_W \otimes_{\wh R_W} \wh R_W \otimes_{R_W} R_{W'}'
= F_W \otimes_{\wh R_W} \wh R_{W'}'
= F_W \otimes_{\wh R_W} \prod \wh R'_{W'_j}
= \prod S^{-1}\wh R'_{W'_j}
= \prod F'_{W'_j}$.
\end{proof}

The above definition of $\wh R_W$ requires that $W$ is non-empty.  But if the
closed fiber $X$ of $\wh X$ is irreducible with generic point $\eta$, then we
define $\wh R_\varnothing$ to be $\wh R_\eta$.  In this situation note that
the equivalence in Proposition~\ref{connec_conds} still holds.

The fields $F_W$ arise in particular when considering a finite morphism $f:\wh
X \to \wh X'$ of projective normal $T$-curves, corresponding to a finite field
extension $F/F'$.
If $U$ is a non-empty connected open subset of the closed fiber $X'$ of $\wh X'$, then
by Proposition~\ref{connec_conds}(\ref{extension structure}) the tensor product
$F'_U \otimes_{F'} F$ is a product of finitely many fields, each of them of
the form $F_W$ for some open subset $W \subseteq X$.
Namely, these sets $W$ are the connected components of $f^{-1}(U) \subseteq X$.
Here the sets $W$ can each meet more than one irreducible component of $X$,
even if $U$ meets just one irreducible component of $X'$.

In the other direction, consider a finite separable extension $E_U$ of the field $F'_U$, where $F' = K(x)$ is the function field of the projective line $\wh X' = \mbb P^1_T$; where $U = \mbb A^1_k$; and where $F'_U$ is the patching field associated to $U$ on the closed fiber $X'$ of $\wh X'$.
Then according to the second part of the next result, there is a finite field
extension $F/F'$, corresponding to a finite morphism $f:\wh X \to \wh X'$ of
projective normal $T$-curves, such that $F \otimes_{F'} F'_U$ is
$F'_U$-isomorphic to $E_U$.
Hence $W := f^{-1}(U)$ is a connected affine open subset of the closed fiber
$X$ of $\wh X$, and $F_W$ is  $F'_U$-isomorphic to the given field $E_U$.

Before stating the proposition, recall some notation.  Let
$\wh X$ be a normal model of a one-variable function field $F$
over the complete discretely valued field $K$, and $X$ the closed fiber of $\wh X$.  
For each point $P \in X$, 
let $R_P$ be the local ring of $\wh X$ at $P$.  
Its completion $\wh R_P$ is a domain (\cite[page~88]{HH:FP}), with
fraction field denoted by $F_P$.
Each height one prime $\wp$ in $\wh R_P$ that contains the uniformizer $t$ of $K$ defines
a \textit{branch} of $X$ at $P$, lying on some irreducible component of $X$.
The $t$-adic
completion $\wh R_\wp$ of the local ring $R_\wp$ of $\wh R_P$ at $\wp$
is a complete discrete valuation ring; its fraction field is denoted by
$F_\wp$.  Hence $F_\wp$ contains $F_P$, and is its completion.  
The field $F_\wp$ also contains $F_U$ if $U$ is an
irreducible open subset of $X$ such that $P \in \bar U \smallsetminus U$.
If $\wh X'$ is another curve with function field $F'$ we will write
$\wh R'_{\wp}$, $\wh R'_P$, $F'_{\wp}$ etc.\ for the analogously defined objects. 

\begin{prop} \label{descend_exten}
Let $U = \mbb A^1_k$, let $P$ be the point $(x=\infty)$ on $\mbb P^1_k$,
and let $\wp$ be the unique branch of $\mbb P^1_k$ at $P$,
where $X=\mbb P^1_k$ is viewed as the closed fiber of $\wh X = \mbb P^1_T$.

\renewcommand{\theenumi}{\alph{enumi}}
\renewcommand{\labelenumi}{(\alph{enumi})}
\begin{enumerate}

\item \label{descend_branch}
For every finite separable field extension $E_\wp$ of $F_\wp$, there is a finite separable field extension $E_P$ of $F_P$ such that $E_P \otimes_{F_P} F_\wp \cong E_\wp$ over $F_\wp$.

\item \label{descend_big_patch}
For every finite separable field extension $E_U$ of $F_U$, there is a finite separable field extension $F'$ of $F:=K(x)$ such that $F' \otimes_F F_U \cong E_U$ over $F_U$.
Moreover if $\wh X'$ is the normalization of $\wh X$ in $F'$, with closed fiber $X'$ and associated morphism
$\pi: \wh X' \to \wh X$, then $F'\otimes_F F_U \cong F'_{U'}$ over $F_U$, where 
$U' = \pi^{-1}(U) \subset X'$ is connected. 
\end{enumerate}
\end{prop}

\begin{proof}
(\ref{descend_branch})
Since $F_\wp$ is the $\wp$-adic completion of $F_P$, this follows from Krasner's Lemma (\cite{La}, Prop.~II.2.3).

(\ref{descend_big_patch})
The tensor product $E_U \otimes_{F_U} F_\wp$ is a finite direct product $\prod_i E_{\wp,i}$ of finite separable field extensions $E_{\wp,i}$ of $F_\wp$.  By part~(\ref{descend_branch}), for each $i$ there is a finite separable field extension $E_{P,i}$ of $F_P$ such that $E_{P,i} \otimes_{F_P} F_\wp$ is isomorphic to $E_{\wp,i}$ over $F_\wp$.  We thus have an isomorphism of separable $F_\wp$-algebras
\[E_U \otimes_{F_U} F_\wp \to \bigl(\prod_i E_{P,i}\bigr) \otimes_{F_P} F_\wp.\]
Applying the patching assertion Theorem~7.1 of \cite{HH:FP} (in the context of Theorem~5.9 of that paper), we obtain a finite separable $F$-algebra $F'$ that induces $E_U$ over $F_U$ and induces $\prod_i E_{P,i}$ over $F_P$, compatibly with this isomorphism.  Since $E_U$ is a field, so is $F'$.

For the last part, $F'\otimes_F F_U \cong \prod_{j=1}^n F'_{U'_j}$ by
Proposition~\ref{connec_conds}(\ref{extension structure}), 
where $U'_1,\dots,U'_n$ are the connected components of $U'$.  But 
$F'\otimes_F F_U \cong E_U$ is a field.  So $n=1$ and the assertion follows.  
\end{proof}

As the above proof shows, Proposition~\ref{descend_exten} holds more generally for non-empty affine open subsets $U$ of the closed fiber $X$ of a smooth projective $T$-curve $\wh X$, together with the set of points $P \in X$ in the complement of $U$.  For this, one cites Theorem~7.1 of \cite{HH:FP} in the context of Theorem~5.10 of that paper.

In the case that $T$ is an {\em equal characteristic} complete discrete valuation ring, an analog of Proposition~\ref{descend_exten}(\ref{descend_big_patch}) for a finite extension of $F_P$ appeared in \cite[Lemma~3.8]{HHK:Weier}.  For a more general choice of $T$, 
there is the following weaker result, which nevertheless will suffice for our purposes below (in Corollary~\ref{extension blowup both cases}):

\begin{prop} \label{small patch blowup}
Let $\wh X$ be a projective normal $T$-curve, let $P$ be a closed point on the closed fiber $X$, and let $E$ be a finite separable extension of $F_P$.  Let $S$ be the integral closure of $\wh R_P$ in $E$, and let $\wh V^* \to \Spec(S)$ be a birational projective morphism, with $\wh V^*$ normal.  Then
there exist normal schemes $\wh V$, $\wh Z$, and $\wh Y$ and a commutative diagram
\[
\xymatrix @R=.7cm {
\wh V \ar[r] \ar[d] & \wh V^* \ar[r] & \Spec(S) \ar[d] \\
\wh Z \ar[rr] \ar[d] & & \Spec(\wh R_P)  \ar[d]\\
\wh Y \ar[rr] & & \wh X
}
\]
of $T$-schemes, where the horizontal maps are birational projective morphisms that are isomorphisms away from (the inverse image of) $P$; the bottom half of the diagram is a pullback square; and the morphism $\wh V \to \wh Z$ is finite.
\end{prop}

\begin{proof}
Let $L$ be the Galois closure of $E$ over $F_P$, let $G=\Gal(L/F_P)$, and let $R$ be the integral closure of $\wh R_P$ in $L$.  Let $H = \Gal(L/E)$ and let $\wh W^*$ be the normalization of $\wh V^* \times_S R$.
It suffices to prove the assertion with $E$, $S$, and $\wh V$ replaced by $L$, $R$, and $\wh W^*$, provided we also show that the $G$-action on $\Spec(R)$ lifts to a $G$-action on the asserted space $\wh W$.  Namely, we can then take $\wh V = \wh W/H$.  So we now assume that $E$ is Galois over $F_P$.

The Galois group $G=\Gal(E/F_P)$ acts on the (isomorphism classes of) birational projective morphisms to $\Spec(S)$.
Consider the fiber product of the morphisms in the orbit of $\wh V^* \to \Spec(S)$, and let 
$\wh V$ be the irreducible component that dominates $\Spec(S)$.  
Then $\wh V$ is normal since each $G$-conjugate of $\wh V^*$ is; and $\wh V \to \Spec(S)$ is a $G$-stable birational projective morphism that factors through $\wh V^* \to \Spec(S)$.  Thus
the action of $G$ on $\Spec(S)$ lifts to an action of $G$ on $\wh V$.  Let $\wh Z$ be the quotient of $\wh V$ by $G$.  Then the birational projective morphism $\wh V \to \Spec(S)$ descends to a
birational projective morphism $\wh Z \to \Spec(\wh R_P)$.
That is, we obtain a commutative diagram
\[
\xymatrix @R=.7cm {
\wh V \ar[r] \ar[d] & \wh V^* \ar[r] & \Spec(S) \ar[d] \\
\wh Z \ar[rr] & & \Spec(\wh R_P)
}
\]
whose vertical arrows are each finite and $G$-Galois, with generic fiber corresponding to the $G$-Galois field extension $E/F_P$.

The bottom horizontal morphism is a composition of blowups and blowdowns, centered at $P$ and at points on exceptional divisors lying over $P$.
We may perform the corresponding blowups and blowdowns on $\wh X$, 
observing inductively that at each step the spaces mapping to $\Spec(\wh R_P)$ and to $\wh X$ have the same exceptional divisors (fibers over $P$), and that generators of the local ring at a closed point over $P \in \wh X$ also generate the local ring at the corresponding closed point over $P \in \Spec(\wh R_P)$.
This process yields a pullback diagram 
\[
\xymatrix @R=.7cm {
\wh Z \ar[rr] \ar[d] & & \Spec(\wh R_P)  \ar[d]\\
\wh Y \ar[rr] & & \wh X
}
\]
where the bottom horizontal map is a birational projective morphism which is 
an isomorphism away from $P \in \wh X$.  This gives the desired conclusion. 
\end{proof}

%%%%%%%%%%%%%%%%%
\subsection{Patching} \label{patching section}
%%%%%%%%%%%%%%%%%

Below, $\wh X$ is a (projective) normal model of a one-variable function field $F$
over the complete discretely valued field $K$, and $X$ is the closed fiber of $\wh X$.  As in Section~\ref{setup}, for
each point $P$ on the closed fiber $X$ of $\wh X$ we have an associated
complete local domain $\wh R_P$ with fraction field $F_P$; and for each
non-empty connected Zariski strict open subset $U$ of $X$ we may consider the domain $\wh R_U$ and its fraction field $F_U$. For $P \in \mc P$ and $\wp$ a branch of $X$ at $P$, we also have the complete discrete valuation ring $\wh R_\wp$ and its fraction field $F_\wp$.

Consider a non-empty finite subset $\mc P \subset X$.  Let $\mc W$ be the set of
connected components of the complement of $\mc P$; each of these connected components $U \in \mc W$ is a strict open subset of $X$.
The set of all the branches of $X$ at points of $\mc P$ is
denoted by $\mc B$.

If $U \subset U'$ are connected strict open subsets of $X$, then
$F_{U'} \subset F_U$.  For $P \in U$, there is also an inclusion $F_U \subset F_P$; and if $\wp$ is a branch at $P$ lying on the closure of $U$, then there are inclusions $F_P, F_U \subset F_\wp$.  These containments are compatible, as $U,U'$ vary.

The next result generalizes \cite[Theorem~6.4]{HH:FP} and \cite[Proposition~2.3(a)]{HHK:Weier}.

\begin{prop} \label{big patch ok}
Let $\mc P$ be a non-empty finite set of closed points of $\wh X$, let
$\mc W$ be the set of connected components of the complement $V$ of $\mc P$ in the closed fiber $X$, and let
$\mc B$ be the
set of branches of $X$ at the points of $\mc P$.
For $Q \in \mc P$, let $\wh R_Q^\circ$ be the subring of $F_Q$ that consists of the elements that lie in $\wh R_\wp$ for each branch $\wp$ of $X$ at $Q$.
Then patching holds for the following injective diamonds.
\renewcommand{\theenumi}{\alph{enumi}}
\renewcommand{\labelenumi}{(\alph{enumi})}
\begin{enumerate} 
\item \label{big patch ok fields}
$F_\bullet = (F\,\leq\, \prod_{U \in \mc W}
  F_U, \ \prod_{Q \in \mc P} F_Q\,\leq\, \prod_{\wp \in \mc B} F_\wp)$.
\item \label{big patch ok rings}
$R_\bullet = (R_V \,\leq\, \prod_{U \in \mc W} \wh R_U, \ \prod_{Q \in \mc P} \wh R_Q^\circ \,\leq\, \prod_{\wp \in \mc B} \wh R_\wp)$.
\end{enumerate}
\end{prop}

\begin{proof}
For short write $F_1 = \prod_{U \in \mc W}
F_U$,  $F_2 = \prod_{Q \in \mc P} F_Q$, $F_0 = \prod_{\wp \in \mc B} F_\wp$,
$R_1 = \prod_{U \in \mc W} \wh R_U$,
$R_2^\circ = \prod_{Q \in \mc P} \wh R_Q^\circ$, 
and $R_0 = \prod_{\wp \in \mc B} \wh R_\wp$.
Thus $F_\bullet = (F\,\le \, F_1,F_2\, \le \, F_0)$
and $R_\bullet = (R_V \,\leq\, R_1, R_2^\circ\, \le \, R_0)$.
The proofs for the two diamonds are similar.  We begin with $F_\bullet$.

First consider the special case that the set $\mc P$  meets each irreducible component of $X$ non-trivially.
By \cite[Proposition~3.3]{HHK:H1}, there is then a finite morphism $f:\wh X \to \mbb P^1_T$ such that $\mc P$ is the fiber over the point $\infty$ on the closed fiber $\mbb P^1_k$ of $\mbb P^1_T$.  Thus $\mc W$ is the set of connected components
of $f^{-1}(U')$, where $U' = \mbb A^1_k$.
Let $F'$ be the function field of $\mbb P^1_T$, so that the map $f$
gives a finite field extension $F/F'$.  Let $F'_\bullet
= (F'\,\le \, F'_1,F'_2\, \le \, F'_0)$ be defined
analogously as above for the curve $\mbb P^1_T$, with $\mc P' = \{\infty\}$ and $\mc W' = \{U'\}$. By \cite[Theorem~5.9]{HH:FP}, patching holds for $F'_\bullet$.  Since $F/F'$ is a
finite field extension, Lemma~\ref{corestriction} implies that patching holds for
$(F, F'_1\otimes_{F'} F, F'_2 \otimes_{F'} F, F'_0 \otimes_{F'} F)$.
The proposition now follows from the assertion that
$F_i = F'_i \otimes_{F'} F$, which holds for $i=1$ by
Proposition~\ref{connec_conds}(\ref{extension structure}) and for
$i=0,2$ by \cite[Lemma~6.2]{HH:FP} (enlarging the set $S'$ there if necessary).

Now consider the case that $\mc P$ does not meet each irreducible component of
$X$.  Since $\mc P$ is non-empty, not every irreducible component of $X$ is
disjoint from $\mc P$.  So by Remark~\ref{RW remark}(\ref{contraction remark}), we may contract
the components that are disjoint from $\mc P$,
via a proper birational morphism $\pi:\wh X \to \wh
Y$.  The set $\mc P$ maps bijectively to its image in $\wh Y$, with $\pi$
inducing an isomorphism between $F_Q$ (taken on $\wh X$) and $F_{\pi(Q)}$
(taken on $\wh Y$), for $Q \in \mc P$.
Similarly, $\pi$ induces an isomorphism between $F_\wp$ and $F_{\pi(\wp)}$ 
for $\wp \in \mc B$.
Moreover for $U \in \mc W$, the morphism $\pi$ induces an isomorphism between
$F_U$ and $F_{\pi(U)}$, by
Remark~\ref{RW remark}(\ref{contraction remark}), since these are the fraction fields of $\wh
R_U$ and $\wh R_{\pi(U)}$ (where the rings are taken on $\wh X$ and $\wh Y$ respectively).  Thus
the assertion for $\wh X$ is equivalent to the assertion for $\wh Y$, which
holds by the first case.  This completes the proof of patching for
$F_\bullet$.

Next we turn to patching for the diamond $R_\bullet$.  As above, we are reduced to the case that the set $\mc P$ meets each irreducible component of $X$
non-trivially, so that there is a finite morphism $f:\wh X \to \mbb P^1_T$ such that $\mc P = f^{-1}(\infty)$.
With notation as above, consider the analogous diamond $R'_\bullet = (R'_{U'} \leq R'_1, R_2'^\circ \leq R_0')$ taken with respect to $\mc P'=\{\infty\}$, where $R'_1 = \wh R_{U'}'$.
Note that $R_1 = \wh R_V = R_1' \otimes_{R_{U'}'} R_V$ by
Proposition~\ref{connec_conds}(\ref{patch ring structure},\ref{extension structure}).
Together with \cite[Lemma~6.2]{HH:FP}, this implies that 
$R_\bullet = R'_\bullet \otimes_{R_{U'}'} R_V$.

Also, $R_V$ is a finitely generated free module over $R_{U'}'$ by
\cite[Proposition~II.3.2.5(ii)]{Bo:CA},
using that
the finitely generated module $R_V/tR_V$ over the principal ideal domain $R_{U'}'/(t) = k[x]$
is torsion-free and hence free and also that
$(t)$ is the Jacobson radical of $R_{U'}'$.  Now intersection holds for
$R'_\bullet$ because $R_{U'}' \subseteq R_1' \cap R_2'^\circ \subseteq R_1' \cap F_1' \cap F_2' = R_1' \cap F' = R_{U'}'$, and factorization holds for $\GL_n(R'_\bullet)$  by \cite[Proposition~2.3(a)]{HHK:Weier}.  That is, patching holds for $R'_\bullet$.  Hence it also holds for $R_\bullet = R'_\bullet \otimes_{R_{U'}'} R_V$, by Lemma~\ref{corestriction}.
\end{proof}

The next result generalizes
\cite[Corollary~2.4, Theorem~3.1(a), and Corollary~3.3(a)]{HHK:Weier}, the
second of which is a form of the Weierstrass Preparation Theorem.  
It follows easily from Proposition~\ref{big patch ok}(\ref{big patch ok rings}) in the same way that those three previous results followed from
\cite[Proposition~2.3(a)]{HHK:Weier}. See also \cite[Theorem~3.6]{Cuong}.

\begin{cor} \label{Weier cor}
Let $\mc W$ be as in Proposition~\ref{big patch ok}.
\renewcommand{\theenumi}{\alph{enumi}}
\renewcommand{\labelenumi}{(\alph{enumi})}
\begin{enumerate}
\item \label{Weier patches}
Suppose that for every $U \in \mc W$ we are given an element $a_U \in F_U^\times$.  Then there exist $b \in F$ and elements $c_U \in \wh R_U^\times$ such that $a_U = bc_U \in F_U^\times$ for all $U \in \mc W$.
\item \label{Weier}
If $U \in \mc W$ and $a \in F_U$ then there exist $b \in F$ and $c \in \wh R_U^\times$ such that $a=bc \in F_U$.  More generally, if $a \in F_U$ and $n$ is a positive integer that is not divisible by the characteristic of the residue field $k$ of $T$, then there exist $b \in F$ and $c \in \wh R_U^\times$ such that $a=bc^n \in F_U$.
\end{enumerate}
\end{cor}

The next result is analogous to Proposition~\ref{big patch ok}(\ref{big patch ok fields}), with $F$ replaced by $F_W$, for $W \subset X$.

\begin{prop}\label{local patching}
Let $W \subseteq X$ be a connected open subset of $X$.  Let $\mc
P$ be a non-empty finite set of closed points of $W$; let $\mc W$ be the
set of connected components of the complement of $\mc P$ in $W$;
and let $\mc B$ be the set of branches of $W$ at the points of
$\mc P$.  Then patching holds for the injective diamond
\(\displaystyle F_{W\bullet} =
\left(F_W \le \prod_{U \in \mc W} F_U,
\prod_{Q \in \mc P} F_Q \le \prod_{\wp \in \mc B} F_\wp\right).\)
\end{prop}

\begin{proof}
If $W=X$ then the assertion is given by Proposition~\ref{big patch ok}(\ref{big patch ok fields}).  So
assume that $W$ is strictly contained in $X$. After blowing down all
irreducible components of $X$ that do not intersect $W$ as in Remark~\ref{RW remark}(\ref{contraction remark}), we may assume that
the closure of $W$ is $X$. Let
$\til{\mc P}$ be the complement of $W$ in its closure $X$.
Also let $\til{\mc B}$ be the set of branches at the points of $\til{\mc P}$.  By 
Proposition~\ref{big patch ok}(\ref{big patch ok fields}), patching holds for the diamond $\til
F_\bullet  = (F\leq  F_W,\prod_{Q \in \til{\mc P}} F_Q\leq \prod_{\wp \in \til {\mc B}} F_\wp)$.

Let $\wh{\mc P}$ be the disjoint union $\til{\mc
P} \sqcup \mc P$ and let $\wh{\mc B}$ be
the set of branches at the points of $\wh{\mc P}$. Thus 
$\wh{\mc B} = \til{\mc
  B} \sqcup \mc B$. Notice that the set of connected components
of the complement of $\wh{\mc P}$ in $X$ is the set of connected components of
the complement of ${\mc P}$ in $W$, i.e., ${\mc W}$.
Again, patching
holds for the diamond
$\wh F_\bullet  = (F\leq \prod_{U \in {\mc W}} F_U,\, \prod_{Q \in \wh{\mc P}}
F_Q\leq\prod_{\wp \in \wh{\mc B}} F_\wp)$ by Proposition~\ref{big patch ok}(\ref{big patch ok fields}).

The general linear groups for the various products of fields form the following diagram:
\[
\xymatrix @1@=0.1pt@M=0.1pt@R=.6cm
{
& & & {\displaystyle\prod_{\wp \in \til{\mc B}} \GL_n(F_\wp)}
& & 	{\displaystyle\prod_{\wp \in \mc B} \GL_n(F_\wp)} \\
& & \!\!\!\!\!{\displaystyle\prod_{Q \in \til{\mc P}} \GL_n(F_Q)} \ar[ru]
& & 	\!\!\!\!\!{\displaystyle\prod_{U \in \mc W} \GL_n(F_U)} \ar[lu] \ar[ru]
& & {\displaystyle\prod_{Q \in \mc P} \GL_n(F_Q)} \ar[lu] \\
& & & & & \GL_n(F_W) \ar[lu] \ar[ru] \\
& & & & \GL_n(F) \ar[ru] \ar[lluu]
}\]
As noted above, patching holds for the diamonds $\til F_\bullet$ and $\wh F_\bullet$;
and so
by Theorem~\ref{patching prop thm}(\ref{patching and matrix factorization}),
factorization and intersection hold for
the diamonds of groups
\[G_\bullet:=\GL_n(\til F_\bullet) = \bigl(\GL_n(F) \le
\GL_n(F_W),
\prod_{Q \in \til{\mc P}} \GL_n(F_Q)
\le \prod_{\wp \in \til{\mc B}} \GL_n(F_\wp)\bigr),\]
\[G_\bullet ':=\GL_n(\wh F_\bullet) = \bigl(\GL_n(F) \le
\prod_{U \in \mc W} \GL_n(F_U),
\prod_{Q \in \wh{\mc P}} \GL_n(F_Q)
\le \prod_{\wp \in \wh{\mc B}} \GL_n(F_\wp)\bigr).\]
Proposition~\ref{subfactor} (parts \ref{subfactor 1} and \ref{subfactor 4})
and
Lemma~\ref{factor ordering}
yield factorization and intersection for
\[\GL_n(F_{W\bullet}) = \bigl(\GL_n(F_W) \le
\prod_{U \in \mc W} \GL_n(F_U),
\prod_{Q \in {\mc P}} \GL_n(F_Q)
\le \prod_{\wp \in {\mc B}} \GL_n(F_\wp)\bigr).\]
That is, patching holds for the diamond $F_{W\bullet}$, as desired.
\end{proof}

The next result, which answers a question posed by Yong Hu, permits patching on the exceptional divisor of a blow-up $f:\wh X \to \wh Y$.  Alternatively, we can view $f$ as a blow-down, in which a non-empty connected union $V$ of some but not all of the irreducible components of the closed fiber $X \subset \wh X$ are contracted to a point $P \in Y \subset \wh Y$ (cf.\ Remark~\ref{RW remark}(\ref{contraction remark})).

\begin{prop} \label{blowup patch}
Let $f:\wh X \to \wh Y$ be a proper birational morphism of projective normal $T$-curves, having closed fibers $X,Y$ respectively.  Let
$P \in Y$ be a closed point, let
$V \subset X$ be the
inverse image of $P$ in $X$, and let $\til Y \subseteq X$ be the proper transform of $Y$.
Suppose that $\dim(V)=1$, and that $f$ restricts to an isomorphism
$\wh X \smallsetminus V \to \wh Y \smallsetminus \{P\}$.
Choose a finite collection of closed points $\mc
P$ in $V$ that includes all the points of $V \cap \til Y$.
Let $\mc W$ be the set of connected components of
$V \smallsetminus \mc P$, and let $\mc B$ be the set of
branches at the points in $\mc P$ along the components of $V$.
Then
patching holds for the injective diamond
\(\displaystyle F_{P\bullet} = \left( F_P\le \prod_{Q \in \mc P} F_Q,
\prod_{U \in \mc W} F_U\le \prod_{\wp \in \mc B} F_\wp\right),\)
with respect to the natural inclusions.
\end{prop}

\begin{proof}
First observe that there are natural inclusions of $F_P$ into $F_U$ and $F_Q$, for
$U \in \mc W$ and $Q \in \mc P$.  Namely, the natural morphism $\Spec(\wh R_U) \to \wh X$
factors through $\wh X_P := f^{-1}(\Spec(\wh R_P))$, the pullback of $\wh X \to \wh Y$ via $\Spec(\wh R_P) \to \wh Y$.  Since $\wh X \to \wh Y$ is birational, so is
$\wh X_P \to \Spec(\wh R_P)$, and hence the function field of $\wh X_P$ is $F_P$.
The morphism $\Spec(\wh R_U) \to \wh X_P$ induces a homomorphism of function fields in the other direction, $F_P \to F_U$, which is necessarily an inclusion.  The case of $F_P \to F_Q$ is similar.

Let $\til{\mc W}$ be the set of connected components of $Y \smallsetminus\{ P\}$.
Let $\til{\mc B}$ be the set of branches of $Y$ at
$ P$.
Via $f$, we may identify $X \smallsetminus V$ with its isomorphic image $Y
\smallsetminus \{P\}$.  We may thus regard the elements of $\til{\mc W}$ as open subsets of $X$, and the elements of $\til{\mc B}$ as branches of $\til Y$.
Write $\wh {\mc W}$ for the disjoint union $\til{\mc W} \sqcup \mc W$. The set of points of $X$ that lie in no
element of $\wh {\mc W}$ is exactly $\mc P$; let $\wh {\mc B}$ be the set of
branches of $X$ at points of $\mc P$. Thus $\wh{\mc B}$ equals the
disjoint union $\til{\mc B}
\sqcup \mc B$. Note that at the points of $\mc P$, some
of the branches of $X$ are elements of $\mc B$ and some are in $\til{\mc B}$,
depending on whether the branches lie on a component of $V$ or of $\til Y$.

We may now consider the associated diagram of groups:
\[
\xymatrix @1@=0.1pt@M=0.1pt@R=.6cm
{
& & & {\displaystyle\prod_{\wp \in \til{\mc B}} \GL_n(F_\wp)}
& & 	{\displaystyle\prod_{\wp \in \mc B} \GL_n(F_\wp)} \\
& & \!\!\!\!\!{\displaystyle \prod_{U \in \til{\mc W}} \GL_n(F_U)} \ar[ru]
& & 	\!\!\!\!\!{\displaystyle\prod_{Q \in \mc P} \GL_n(F_Q)} \ar[lu] \ar[ru]
& & {\displaystyle\prod_{U \in \mc W} \GL_n(F_U)} \ar[lu] \\
& & & & & \GL_n(F_P) \ar[lu] \ar[ru] \\
& & & & \GL_n(F)  \ar[ru] \ar[lluu]
}\]
By Proposition~\ref{big patch ok}(\ref{big patch ok fields})
and Lemma~\ref{factor ordering},
patching holds for the diamonds
\[\til F_\bullet = (F \le \prod_{U \in {\til{\mc W}}} F_U ,F_P\le \prod_{\wp \in \til{\mc B}} F_\wp),\ \ \
\wh F_\bullet = (F \le \prod_{Q \in \mc P} F_Q, \prod_{U \in \wh {\mc W}} F_U \le \prod_{\wp \in \wh {\mc B}} F_\wp).\]
That is, factorization and intersection hold for
the diamonds of groups
$G_\bullet:=\GL_n(\til F_\bullet)$ and $G_\bullet':=\GL_n(\wh F_\bullet)$.
Parts~(\ref{subfactor 1}) and~(\ref{subfactor 4}) of
Proposition~\ref{subfactor} imply that the diamond $\GL_n(F_{P\bullet})$ satisfies
factorization and intersection; i.e.\
patching holds
for $F_{P\bullet}$, as desired.
\end{proof}

\begin{example} \label{blow-up example}
Let $T = k[[t]]$ and let $\wh Y$ be the projective $y$-line $\mbb P^1_T$, with closed fiber $Y = \mbb P^1_k$.  Let $P$ be the point $y=0$ on $Y$, with complete local ring $\wh R_P = k[[y,t]]$.
Consider the blow-up $\wh X \to \wh Y$ of $\wh Y$ at $P$.  The exceptional divisor $V$ is a copy of the $x$-line over $k$, with $x=0$ at the point ${P'}$ of $\wh X$ where $V$ meets the proper transform of $Y$.  The complete local ring $\wh R_{P'}$ is 
$k[[y,t,x]]/(t-xy) = k[[x,y]]$.  Writing $W$ for the complement of ${P'}$ in $V$, the ring 
$\wh R_W$ is the $t$-adic completion of $k[[y,t]][x^{-1}]/(x^{-1}t-y)$, which is naturally isomorphic to $k[x^{-1}][[t]]$ (with $y=x^{-1}t$).
The unique branch $\wp$ of $V$ at ${P'}$ has associated ring
$\wh R_\wp = k((x))[[y]]$, which contains $\wh R_{P'}$ and $\wh R_W$.  The intersection of these two rings in $\wh R_\wp$ is $\wh R_P$.
The respective fraction fields are
$F_P = k((y,t))$, $F_{P'} = k((x,y))$, $F_W = \Frac\bigl(k[x^{-1}][[t]]\bigr)$, and $F_\wp = k((x))((y))$.  They
satisfy the intersection condition 
$F_P = F_{P'} \cap F_W \subset F_\wp$, and they also satisfy factorization for $\GL_n$.  This example is a twisted form of the example given in \cite{HH:FP} after Theorem~5.9 there.  It is also related to the situations discussed in \cite[Section~1]{PN} and \cite{BBT}.
\end{example}

The next corollary,
which will be useful in Section~\ref{csa section},
is a variant of the previous proposition.  Here we blow up $\Spec(S)$ for some two-dimensional complete ring $S$ that need not be of the form $\wh R_P$, but instead can be a finite extension of some $\wh R_P$ or some $\wh R_W$.  
In this situation, we can again consider fields of the form $F_Q$, $F_U$, and $F_\wp$, associated to this blowup; the previous definitions carry over mutatis mutandis to the case of any two-dimensional normal scheme whose closed points all lie on a connected curve.  

\begin{cor} \label{extension blowup both cases}
Let $\wh X$ be a projective normal $T$-curve with closed fiber $X$, and let $\xi$ be
either a closed point $P \in X$ or a connected affine open subset $W \subset X$.
Let $E$ be a finite separable extension of $F_\xi$, let
$S$ be the integral closure of $\wh R_\xi$ in $E$, and
let $\til\xi$ be the inverse image of $\xi$ under $\Spec(S) \to \Spec(\wh R_\xi)$.
Let $D$ be a divisor on $\Spec(S)$.
Then there exist a birational projective morphism $\pi:\wh V \to \Spec(S)$ and a finite set $\mc P$ of closed points of
$V := \pi^{-1}(\til\xi)$
such that the following hold:
\begin{enumerate}[\ \ (i)]
\item \label{bl up normal}
$\wh V$ is a normal scheme.
\item \label{ncd}
$D':=\pi^{-1}(D)$ is a normal crossing divisor on $\wh V$.
\item \label{crossing points}
$\mc P$ contains all the points of $V$ where $V \cup D'$ is not regular, and it meets every connected component of the exceptional locus of $\pi$.
\item \label{patching on bl up}
Let $\mc W$ be the set of connected components of $V \smallsetminus \mc P$,
and let $\mc B$ be the collection of branches of $V$ at the points of $\mc P$.
If $\mc P, \mc W$ are non-empty, then
patching holds for the injective diamond
\(\displaystyle E_\bullet =  (E \le \prod_{Q \in \mc P} F_Q,
\prod_{U \in \mc W} F_U \le
\prod_{\wp \in \mc B} F_\wp).\)
\end{enumerate}
\end{cor}

\begin{proof}
{\em Case~I:} $\xi = P \in X$.

Since $\Spec(S)$
is two-dimensional, excellent, and normal, by \cite{Abh} and \cite{Lip} there is a birational projective morphism (viz.\ a composition of blowups)
$\pi^*:\wh V^* \to \Spec(S)$ such that $\wh V^*$ is regular and
$D^* := (\pi^*)^{-1}(D) \subset \wh V^*$
is a normal crossing divisor.
By Proposition~\ref{small patch blowup}, we obtain a diagram
\[
\xymatrix @R=.7cm {
\wh V \ar[r] \ar[d] \ar@/^1.3pc/[rr]^\pi \ar@/_1pc/[dd]_\alpha & \wh V^* \ar[r] & \Spec(S) \ar[d] \\
\wh Z \ar[rr]^\sigma \ar[d] & & \Spec(\wh R_P)  \ar[d]\\
\wh Y \ar[rr]^\omega & & \wh X
}
\]
with the properties asserted there.  Here $D' := \pi^{-1}(D)$ is a normal crossing divisor on $\wh V$ because $D^*$ is, and since the map $\wh V \to \wh V^*$ is a blow-up.  

Recall that $V := \pi^{-1}(\til P)$, where $\til P \in \Spec(S)$ is the inverse image of $P \in \Spec(\wh R_P)$. 
Let $Y =\omega^{-1}(X)$ be the closed fiber of $\wh Y$, and  
$V_0 := \alpha(V) = \omega^{-1}(P) \subset Y$.
Choose a non-empty finite set $\mc P_0$ of closed points of $V_0$
that contains
the image under $\alpha$ of
the locus where $V \cup D'$ is not regular, and
also contains the points where $V_0$ meets the proper transform of $X$ in $Y$.
Let $\mc P = \alpha^{-1}(\mc P_0) \subseteq V$.
Thus the above
properties~(\ref{bl up normal}),~(\ref{ncd}),~(\ref{crossing points}) hold.
Let $\mc W$ be as in (\ref{patching on bl up}).

If $\pi$ is an isomorphism, then $\mc W$ is empty and we are done.  Otherwise,
$V$ is a curve, $\mc W$ is non-empty, and it remains to show that patching holds for the diamond $E_\bullet$.

Let $\mc W_0$ be the set of connected components of $V_0 \smallsetminus \mc P_0$ 
and let $\mc B_0$ be the set of branches of $V_0$ 
at the points of $\mc P_0$.
Thus the elements of $\mc W$ are the inverse images of the elements of $\mc W_0$, and similarly for $\mc B$ and $\mc B_0$.
By Proposition~\ref{blowup patch},
patching holds for the diamond
\(\displaystyle F_{P\bullet} = \bigl( F_P \le \prod_{Q \in \mc P_0} F_Q,
\prod_{U \in \mc W_0} F_U \le
\prod_{\wp \in \mc B_0} F_\wp  \bigr),\)
taking $F_P$ with respect to $\wh X$ and taking the other fields
with respect to $\wh Y$.

By Proposition~\ref{small patch blowup}, the bottom half of the above diagram 
is a pullback square, with $Z := \sigma^{-1}(P) \to V_0$ 
an isomorphism.
For each $U \in \mc W_0$, with inverse image $U' \subseteq Z$, the natural map $F_U \to F_{U'}$ is an isomorphism; and similarly for $\mc P_0$ and $\mc B_0$.  
The diamond $F_{P\bullet}$ may thus be considered to be taken with respect to $\wh Z$.

The morphism $\wh V \to \wh Z$ in Proposition~\ref{small patch blowup} is finite.  
Thus $\prod_{U \in \mc W} F_U = \prod_{U \in \mc W_0} F_U \otimes_{F_P} E$, 
where $F_U$ on the left is taken with respect to $\wh V$ and $F_U$ on the right is taken with respect to $\wh Z$ (or $\wh Y$); and the analogous isomorphisms hold for the fields $F_Q$ and $F_\wp$.  (These isomorphisms are as in Proposition~\ref{connec_conds}(\ref{extension structure}) and \cite[Lemma~6.2]{HH:FP}, whose statements and proofs carry over to this somewhat more general situation.)
Applying Lemma~\ref{corestriction}(\ref{corestriction fact})
with respect to the field extension $E/F_P$,
we obtain the desired conclusion.

\medskip

{\em Case~II:} $\xi$ is a connected affine open subset $W \subset X$.

Recall that $\til W \subset \Spec(S)$ is the inverse image of $W \subset \Spec(\wh R_W)$.
Choose a non-empty finite subset $\mc P^*$ of closed points of $W$ that contains the image 
under $f:\Spec(S) \to \Spec(\wh R_W)$ of the points where $\til W \cup D$ is not regular, and 
let $\til{\mc P} = f^{-1}(\mc P^*) \subset  \til W \subset \Spec(S)$.  
Write $\til{\mc W}$ for the set of connected components of the complement of 
$\til{\mc P}$ in $\til W$, and $\til{\mc B}$ for the set of branches of 
$\til W$ at the points of $\til{\mc P}$.

Let $\mc W^*$ be the set of connected components of $W \smallsetminus \mc P^*$
and let $\mc B^*$ be the set of branches of $W$ at the points of $\mc P^*$.
Thus the elements of $\til{\mc W}$ are the inverse images under $f$ 
of the elements of $\mc W^*$,
and similarly for $\til{\mc B}$ and $\mc B^*$.  
By Proposition~\ref{local patching}, patching holds for the diamond
\(\displaystyle F_{W\bullet} = \bigl(F_W \le  \prod_{U \in \mc W^*} F_U , \,
\prod_{Q \in \mc P^*} F_Q
\le \prod_{\wp \in \mc B^*} F_\wp  \bigr).\)
Since $f$ is finite, 
$\prod_{U \in \til{\mc W}} F_U = \prod_{U \in \mc W^*} F_U \otimes_{F_W} E$, 
and similarly for the fields $F_Q$ and $F_\wp$,
as at the end of the proof of Case~I.  As in that proof,
patching holds for the diamond 
$\til E_\bullet := (E \le \prod_{U \in \til{\mc W}} F_U, \, 
\prod_{Q \in \til{\mc P}} F_Q \le
\prod_{\wp \in \til{\mc B}} F_\wp)$,
by Lemma~\ref{corestriction}(\ref{corestriction fact}).
That is, factorization and intersection hold for $G_\bullet:=\GL_n(\til E_\bullet)$ for all~$n$ (see Definition~\ref{patching prop def}).

Let $\til{\mc P}' \subseteq \til{\mc P}$ consist of the closed points of $\til W$ 
where $D$ is not a normal crossing divisor, and write 
$\wh{\mc P} = \til{\mc P} \smallsetminus \til{\mc P}'$.  
Thus $\til{\mc P} = \til{\mc P}' \sqcup \wh{\mc P}$.
Our strategy will be to blow up $\Spec(S)$ at the points of $\til{\mc P}'$, obtaining a refined diamond $E_\bullet$; and then to use that patching holds for 
$\til E_\bullet$ and the diamond arising from the exceptional locus to obtain the same for $E_\bullet$ via
Proposition~\ref{subfactor}.

For each $Q \in \til{\mc P}'$, consider
the complete local ring $\wh R_Q$ of $\Spec(S)$ at the point $Q$, with fraction field $F_Q$.
Thus $f(Q) \in \mc P^* \subset W \subset \wh W = \Spec(\wh R_W)$, and 
$F_Q$ is a finite separable extension of $F_{f(Q)}$, viz.\
a factor of $F_{f(Q)} \otimes_{F_W} E$.  Let $D_Q$ be the restriction of $\til W \cup D$ to $\Spec(\wh R_Q)$.

By Case~I, for each $Q \in \til{\mc P}'$, 
there is a birational projective morphism (viz.\ a composition of blowups) 
$\pi_Q':\wh V_Q' \to \Spec(\wh R_Q)$ for which 
conditions~(\ref{bl up normal})-(\ref{patching on bl up})
are satisfied,
with respect to the divisor $D_Q$,
some finite subset ${\mc P}_{\!Q}'$ of $V_Q' := \pi_Q'^{-1}(Q)$,
the associated sets $\mc W_{\!Q}'$ and $\mc B_{\!Q}'$, and the diamond
$F_{Q\bullet}' := (F_Q \le \prod_{Q' \in \mc P_{\!Q}'} F_{Q'}, \,
\prod_{U \in \mc W_{\!Q}'} F_U \le
\prod_{\wp \in \mc B_{\!Q}'} F_\wp)$.
In particular, patching holds for $F_{Q\bullet}'$ by~(\ref{patching on bl up}),
where $\mc P_{\!Q}',\mc W_{\!Q}'$ are non-empty since $Q \in \til{\mc P}'$.
That is, intersection and factorization hold for 
$\GL_n(F_{Q\bullet}')$ for all $n$.
By Lemma~\ref{enlarging diamonds}(\ref{product diamond}), these properties
also hold for $\GL_n(F_\bullet')$, where 
$F_\bullet' = \prod_{Q \in \til{\mc P}'} F_{Q\bullet}'$, 
with the product of diamonds being taken entry by entry.  

Since blowing up is local, we may take the corresponding blowups of $\Spec(S)$ at ideals respectively supported at the points $Q \in \til{\mc P}'$.
We thus obtain a projective birational morphism $\pi:\wh V \to \Spec(S)$ which is an isomorphism away from $\til{\mc P}'$,
and whose pullback under
$\Spec(\wh R_Q) \to \Spec(S)$
may be identified with $\wh V_Q'$, for $Q \in \til{\mc P}'$.
We may similarly regard $\til W \smallsetminus \til{\mc P}'$ 
as contained in $\wh V$.
With respect to these identifications, let
$\mc P' \subset V := \pi^{-1}(\til W)$ be the disjoint union of the sets 
$\mc P'_Q$ for $Q \in \til{\mc P}'$,
and similarly define $\mc W'$ and $\mc B'$.
Note that $\mc P:=\mc P' \sqcup \wh{\mc P}$ contains all the points 
of $V$ at which $V \cup D'$ is not regular,
where $D' = \pi^{-1}(D)$.  

Now properties~(\ref{bl up normal}),~(\ref{ncd}),~(\ref{crossing points}) hold for 
$\pi$ with respect to the divisor $D$ and the set $\mc P$, where~(\ref{ncd}) uses that $D \subset \Spec(S)$ is a normal crossing divisor away from $\til{\mc P}'$.
Let $\mc W$ be the set of connected components of $V \smallsetminus \mc P$, 
and let $\mc B$ be the set of branches of $V$ at the points of~$\mc P$.
Thus $\mc W = \til{\mc W} \sqcup \mc W'$ and 
$\mc B = \til{\mc B} \sqcup \mc B'$.
It remains to show that property~(\ref{patching on bl up}) is satisfied
for the diamond
$E_\bullet := (E \le \prod_{Q \in \mc P} F_Q, \,
\prod_{U \in \mc W} F_U \le
\prod_{\wp \in \mc B} F_\wp)$.  

It was shown above that for any $n$,
intersection and factorization hold for 
\[\til H_\bullet := \GL_n(F_\bullet')
= \bigl(\prod_{Q \in \til{\mc P}'} \GL_n(F_Q) \le \prod_{Q' \in {\mc P}'} \GL_n(F_{Q'}),
\prod_{U \in \mc W'} \GL_n(F_U) \le \prod_{\wp \in \mc B'}  \GL_n(F_\wp)\bigr).\]  
Applying Lemma~\ref{enlarging diamonds}(\ref{extend diamond})
with $A = \GL_n(\prod_{Q\in \wh{\mc P}} F_Q)$, it follows that these two
properties also hold for
the (non-injective) diamond
\[H_\bullet := \bigl(\prod_{Q \in \til{\mc P}} \GL_n(F_Q),\prod_{Q \in {\mc P}} \GL_n(F_Q),
\prod_{U \in \mc W'} \GL_n(F_U), \prod_{\wp \in \mc B'}  \GL_n(F_\wp)\bigr).\]
(Here we use that $\til{\mc P} = \til{\mc P}' \sqcup \wh{\mc P}$ and that $\mc P=\mc P' \sqcup \wh{\mc P}$.)  Consider the injective diamond
\[G'_\bullet := \GL_n(E_\bullet) = \bigl(\GL_n(E) \le \prod_{Q \in {\mc P}} \GL_n(F_Q), \, 
\prod_{U \in \mc W} \GL_n(F_U) \le \prod_{\wp \in \mc B} \GL_n(F_\wp)\bigr).\]
Since $\mc W = \til{\mc W} \sqcup \mc W'$ and 
$\mc B = \til{\mc B} \sqcup \mc B'$,
Proposition~\ref{subfactor} applies
to the diamonds $G_\bullet$, $G'_\bullet$, and $H_\bullet$, with respect to the 
following diagram:
\[
\xymatrix @1@=0.1pt@M=0.1pt@R=.4cm
{
& & & {\displaystyle\prod_{\wp \in \til{\mc B}} \GL_n(F_\wp)}
& & 	{\displaystyle\prod_{\wp \in \mc B'} \GL_n(F_\wp)} \\
& & \!\!\!\!\!{\displaystyle \prod_{U \in \til{\mc W}} \GL_n(F_U)} \ar[ru]
& & 	\!\!\!\!\!{\displaystyle\prod_{Q \in \mc P} \GL_n(F_Q)} \ar[lu] \ar[ru]
& & {\displaystyle\prod_{U \in \mc W'} \GL_n(F_U)} \ar[lu] \\
& & & & & {\displaystyle\prod_{Q \in \til{\mc P}} \GL_n(F_Q)} \ar[lu] \ar[ru] \\
& & & & \GL_n(E)  \ar[ru] \ar[lluu]
}\]
By parts~(\ref{subfactor 5}) and~(\ref{subfactor 6}) of that proposition,
it follows that intersection and factorization 
hold for $G_\bullet'= \GL_n(E_\bullet)$ for all $n$; i.e.\ patching holds for 
\(\displaystyle E_\bullet =  (E \le \prod_{Q \in \mc P} F_Q,
\prod_{U \in \mc W} F_U \le
\prod_{\wp \in \mc B} F_\wp).\)
\end{proof}

\begin{remark}
As the proof of Corollary~\ref{extension blowup both cases} shows, if the conclusion holds for a given choice of
$\wh V$ and of $\mc P \subset V$, and if $\mc P' \subset V$ is any other finite subset of $V$, then one can enlarge $\mc P$ so as to contain $\mc P'$ and still satisfy the conclusion of the corollary.
\end{remark}

%%%%%%%%%%%%%%%%%
\subsection{Factorization for diamonds of groups} \label{fact quad subsect}
%%%%%%%%%%%%%%%%%

In the situation of Section~\ref{patching section}, we can obtain results about factorization  for diamonds that arise from algebraic groups other than just $\GL_n$.  This is useful for obtaining local-global principles for torsors; see Theorem~\ref{factor local-global}.

\begin{prop} \label{diamond group implication}
Let $\wh X$ be a normal model of a one-variable function field $F$ over $K$, and let $G$ be an algebraic group over $F$.
In the notation of Proposition~\ref{big patch ok}, \ref{local patching}, or \ref{blowup patch},
let $\mc P'$ be a finite set of closed points of $X$ 
that contains $\mc P$.
In the context of Proposition~\ref{local patching}, assume also that $\mc P'$ contains $\overline{W} \smallsetminus W$.
Let $\mc W'$ be the set of components of $X \smallsetminus \mc P'$, and let $\mc B'$ be the set of branches of $X$ at the points of $\mc P'$.  
Set $F'_\bullet = (F\leq \prod_{U \in \mc W'} F_U, \ \prod_{Q \in \mc P'} F_Q
\leq \prod_{\wp \in \mc \mc B'} F_\wp)$.
If factorization holds for the diamond $G(F'_\bullet)$ then it holds for $G(F_\bullet)$, $G(F_{W\bullet})$, or $G(F_{P\bullet})$, respectively.
\end{prop}

\begin{proof}
For short write $F'_\bullet = (F\leq F_1',F_2'\leq F_0')$.  We consider each of the three cases in turn.

\textit{Case of Proposition~\ref{big patch ok}.}
Let $n$ be the number of points in $\mc P'$ that are not in $\mc P$.  By induction we are reduced to the case 
that 
$n=1$, since each set $\mc P''$ with $\mc P \subset \mc P'' \subset \mc P'$ is also an allowable finite subset of $X$ in the notation of Proposition~\ref{big patch ok}.
Write $\mc P' = \mc P \sqcup \{P\}$.

Let $W$ be the unique element of $\mc W$ that contains $P$, and let 
$W' = W \smallsetminus \{P\}$.  Consider the associated diamond $F_{W\bullet}$, defined as in Proposition~\ref{local patching} with respect to the set $\{P\} \subset W$.  Thus 
$F_{W\bullet} = (F_W \le \prod_{U \in \mc W'_P} F_U,\, F_P \le \prod_{\wp \in \mc B_P} F_\wp)$, where $\mc W'$ be the set of connected components of $W'$, and where
$\mc B_P$ is the set of branches of $W$ at $P$.
Patching holds for $F_{W\bullet}$ by Proposition~\ref{local patching}; and in particular the intersection property holds for $F_{W\bullet}$ (see Definition~\ref{patching prop def}).   

Let 
\[G_\bullet = G(F_\bullet) = (G(F)\,\leq\, \prod_{U \in \mc W}
  G(F_U), \ \prod_{Q \in \mc P} G(F_Q)\,\leq\, \prod_{\wp \in \mc B} G(F_\wp)),\] 
\[G'_\bullet = G(F'_\bullet) = (G(F)\leq \prod_{U \in \mc W'} G(F_U), \ \prod_{Q \in \mc P'} G(F_Q)
\leq \prod_{\wp \in \mc B'} G(F_\wp)),\]
\[\til H_\bullet = G(F_{W\bullet}) = (G(F_W) \le \prod_{U \in \mc W'} G(F_U),\, G(F_P) \le \prod_{\wp \in \mc B_P} G(F_\wp)).\]
By Lemma~\ref{intersection for groups}, intersection holds for $\til H_\bullet$ since it holds for $F_{W\bullet}$.
Let $A = \prod_{U \in \mc W \smallsetminus \{W\}} G(F_U)$.  Let
$H_\bullet$ be the coordinate-wise product of diamonds  
$\til H_\bullet \times (A,A,1,1)$.  That is, 
\[H_\bullet = (\prod_{U \in \mc W} G(F_U) \le \prod_{U \in \mc W'} G(F_U),\, G(F_P) \le \prod_{\wp \in \mc B_P} G(F_\wp)).\]  
Then intersection holds for 
$H_\bullet$ by Lemma~\ref{enlarging diamonds}(\ref{extend diamond}).
Using Lemma~\ref{factor ordering}, 
the desired conclusion then follows from Proposition~\ref{subfactor}(\ref{subfactor 2}), 
with $G_\bullet$, $G'_\bullet$, and $H_\bullet$ as above, 
with respect to the following refinement diagram:

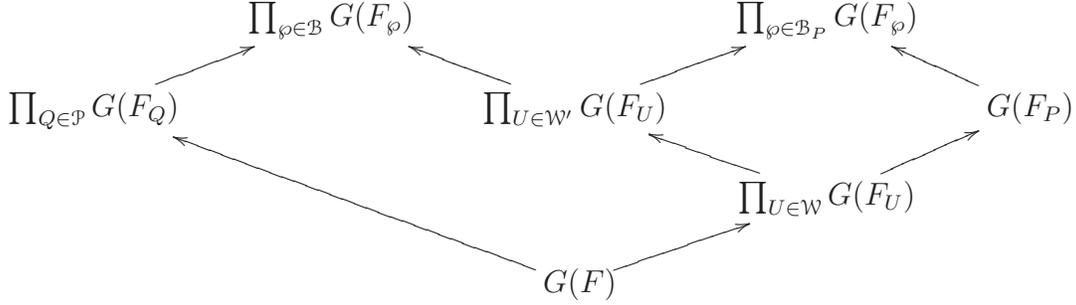
\centerline{\xymatrix @R=.5cm @C=.7cm {
& \prod_{\wp \in \mc B} G(F_\wp) & & \prod_{\wp \in \mc B_P} G(F_\wp) \\
\prod_{Q \in \mc P} G(F_Q) \ar[ru] & & \prod_{U \in \mc W'} G(F_U) \ar[lu]
\ar[ru] & & G(F_P) \ar[lu] \\
& & & \prod_{U \in \mc W} G(F_U) \ar[lu] \ar[ru] \\
& & G(F) \ar[lluu]\!\! \ar[ru]
 }}

\smallskip

\textit{Case of Proposition~\ref{local patching}.}
Let $\mc P''$ be the subset of $\mc P'$ obtained by deleting those points that lie in $W$ but not in $\mc P$,
and consider the corresponding diamond $F''_\bullet$.
By the case
of Proposition~\ref{big patch ok}, it follows that factorization holds for $G(F''_\bullet)$.  So after replacing $\mc P'$ by $\mc P''$, we may assume that $\mc P' \cap W = \mc P$.

Write $\mc P'$ as a disjoint union $\mc P \sqcup  \tilde{\mc P}$.  Thus $\tilde{\mc P}$ is disjoint from $W$.  Let $\tilde{\mc W}$ be the set of connected components of the complement of $\tilde{\mc P}$ in $X$.  Thus $W \in \tilde{\mc W}$, using the hypothesis that $\mc P'$ contains $\overline{W} \smallsetminus W$.
Let $\tilde{\mc B}$ be the set of branches of $X$ at the points of $\tilde{\mc P}$.  Consider the diamond $\tilde F_\bullet = (F\le\tilde F_1,\tilde F_2 \le \tilde F_0)$, where $\tilde F_1,\tilde F_2,\tilde F_0$ are defined analogously to $F_1',F_2',F_0'$.
Write $G_\bullet = G(\tilde F_\bullet)$, $G'_\bullet = G(F'_\bullet)$, and $\til H_\bullet = G(F_{W\bullet})$.  
Letting $A = \prod_{U \in \til {\mc W} \smallsetminus \{W\}} G(F_U)$,
and setting $H_\bullet = \til H_\bullet \times (A,A,1,1)$ as in the first case, 
we then obtain a refinement diagram as in that case.
Using Lemma~\ref{factor ordering}, it follows from
Proposition~\ref{subfactor}(\ref{subfactor 1})
that $H_\bullet$ satisfies factorization.  
By Lemma~\ref{enlarging diamonds}(\ref{extend diamond}), so does 
$\til H_\bullet = G(F_{W\bullet})$.

\smallskip

\textit{Case of Proposition~\ref{blowup patch}.}
We proceed analogously to the case of Proposition~\ref{local patching}.
As in that case, we may assume that $\mc P' \cap V = \mc P$, via the case
of Proposition~\ref{big patch ok}.  Write $\mc P' = \mc P \sqcup  \mc P''$, so that $\mc P'' \subset X$ is disjoint from $V$.  Identifying $X \smallsetminus V$ with $Y \smallsetminus \{P\}$,
we may view $\mc P''$ as a subset of $Y \smallsetminus \{P\}$.
Let $\tilde{\mc P} = \mc P'' \sqcup \{P\} \subset Y$.
Let $\tilde{\mc W}$ be the set of connected components of the complement of $\tilde{\mc P}$ in $Y$, and $\tilde{\mc B}$ the set of branches of $Y$
at the points of~$\tilde{\mc P}$.  Consider the diamond $\tilde F_\bullet = (F\le\tilde F_1,\tilde F_2 \le \tilde F_0)$, where $\tilde F_1,\tilde F_2,\tilde F_0$ are defined as in the previous case (though with respect to $\wh Y$, rather than $\wh X$; note that $F$ is also the function field of $\wh Y$).
Write $G_\bullet = G(\tilde F_\bullet)$, $G'_\bullet = G(F'_\bullet)$, and $\til H_\bullet = G(F_{P\bullet})$.  Setting
$A = \prod_{Q \in \til P \smallsetminus \{P\}} G(F_Q)$
and $H_\bullet = \til H_\bullet \times (A,A,1,1)$ as before, we again obtain a refinement diagram.  As in the previous case, 
using Lemma~\ref{factor ordering}, 
factorization for $H_\bullet$ follows from 
Proposition~\ref{subfactor}(\ref{subfactor 1}).
Lemma~\ref{enlarging diamonds}(\ref{extend diamond}) then implies
that factorization also holds for $\til H_\bullet = G(F_{P\bullet})$.
\end{proof}

\begin{cor} \label{rational fact cor}
Under the hypothesis of Proposition~\ref{big patch ok},~\ref{local patching}, or~\ref{blowup patch},
let $G$ be a rational connected linear algebraic group over $F$.
Then factorization holds for the diamond $G(F_\bullet)$, $G(F_{W\bullet})$, or $G(F_{P\bullet})$, respectively.
\end{cor}

\begin{proof}
Choose a finite set $\mc P'$ of closed points of $X$ that contains
$\mc P$ and all the points where irreducible components of $X$ meet.  In the context of
Proposition~\ref{local patching}, assume also that $\mc P'$ contains $\overline{W} \smallsetminus W$.  With notation as in the statement of Proposition~\ref{diamond group implication},
factorization holds for the diamond $G(F'_\bullet)$
by \cite[Theorem~3.6]{HHK}
(using that the simultaneous factorization condition in \cite{HHK} is equivalent to the
factorization assertion for diamonds; see \cite[Section~2.1.3]{HHK:Hi}).
Hence the conclusion follows by
Proposition~\ref{diamond group implication}.
\end{proof}

In the situation of Corollary~\ref{rational fact cor}, it then follows
from Theorem~\ref{factor local-global} that if $\mc V$ is the
class of $F$-varieties $V$ with a $G$-action such that $G(\prod_{\wp \in \mc B} F_\wp)$ acts transitively on $V(\prod_{\wp \in \mc B} F_\wp)$, then the local-global principle holds for $\mc V$ with respect to $F_\bullet$.
The corresponding local-global assertions also follow for varieties over $F_W$ or $F_P$ (satisfying transitivity over $\prod_{\wp \in \mc B} F_\wp$), with respect to $F_{W\bullet}$ or $F_{P\bullet}$, respectively.

\begin{remark}
Proposition~\ref{diamond group implication} similarly implies the conclusion of Corollary~\ref{rational fact cor} for any linear algebraic group $G$ over $F$ such that factorization holds for $G$ with respect to any choice of a non-empty finite subset $\mc P'$ of $X$ that includes all the points where distinct components meet.  This includes all connected retract rational groups $G$, by \cite{Kra}.  As shown in \cite[Corollary~6.5]{HHK:H1},
if the reduction graph associated to $F$ is a tree then this property also holds for all linear algebraic groups $G$ over $F$ that are rational but not necessarily connected (i.e.\ each connected component is $F$-rational).
\end{remark}

%%%%%%%%%%%%%%%%%
%%%%%%%%%%%%%%%%%
\section{Applications to local-global principles and field invariants} \label{applications}
%%%%%%%%%%%%%%%%%
%%%%%%%%%%%%%%%%%

We now apply the above results in order to obtain applications in the contexts of quadratic forms and central simple algebras.  
These applications, which concern local-global principles and invariants of fields, extend and build on results that appeared in \cite{HHK}, \cite{HHK:H1}, and \cite{HHK:Weier}, as well as in \cite{Leep}, \cite{HuDiv}, \cite{HuLGP}, \cite{HuLaurent}, and \cite{PS:PIu}.  
The fields we consider will be finite extensions of fields $F_P$ and $F_U$, for $P$ a closed point and $U$ a connected open subset of the closed fiber of a curve over a complete discrete valuation ring,
in the notation of Section~\ref{patches}.

%%%%%%%%%%%%%%%%%
\subsection{Applications to quadratic forms} \label{quad form section}
%%%%%%%%%%%%%%%%%

Here we present applications to quadratic forms, 
concerning local-global principles and invariants of fields, especially the $u$-invariant.  
We focus on the fields arising in Section~\ref{patches} and finite separable extensions of these, in particular proving results that generalize and extend assertions in \cite[Section~9.2]{HHK:H1}, and \cite[Section~4.1]{HHK:Weier}
regarding the Witt ring, Witt index, and $u$-invariant.  
As a consequence, we obtain a local-global result for the value of the $u$-invariant (Corollary~\ref{uinv max}).  
Due to \cite{PS:PIu}, this latter result also applies in the case of mixed characteristic $(0,2)$, which is often avoided in quadratic form theory.  
Afterwards we obtain a result (Theorem~\ref{combined thm qf}) concerning the value of the $u$-invariant for finite separable extensions of fields such as $k((x,t))$ and the fraction field of  $k[x][[t]]$, as well as mixed characteristic analogs of such fields.  

We begin with local-global principles, starting with a more general result in
the abstract context of diamonds.  The reader is referred to \cite{Lam} for
basic notions such as isotropic and hyperbolic forms, the Witt ring and
fundamental ideal,  the Witt index, and the $u$-invariant.

If $F_v$ are fields (for $v$ in some index set), we define $W(\prod
F_v):=\prod W(F_v)$ and $I(\prod F_v):=\prod I(F_v)$. Recall that for any
field $E$ and quadratic form $q$ over $E$, $H^1(E,\SO(q))$ classifies quadratic forms of the same dimension
and discriminant as $q$, with $q$ corresponding to the distinguished element
of the Galois cohomology set; see \cite[29.29]{BofInv}.  
(In part~(\ref{ker of map on Witt}) below we write $\mu_2$ rather than $\mbb Z/2 \mbb Z$ as in \cite{HHK:H1}; but these are equivalent since the characteristic is not two.)

\begin{thm} \label{quad form abstract applic}
Suppose that $F_\bullet=(F\leq F_1,F_2\leq F_0)$ is a diamond of rings with $F$ a field of characteristic unequal to two, and each $F_i$ a finite direct product of fields.  
Write $F_i = \prod_{v \in \mc V_i} F_v$ where each $F_v$ is a field. Assume moreover that patching holds for $F_\bullet$ and 
that we have factorization for the diamond
$\SO(nh)(F_\bullet)$
for each $n > 0$, where $h$ is a hyperbolic plane $\<1,-1\>$.
Then:
\renewcommand{\theenumi}{\alph{enumi}}
\renewcommand{\labelenumi}{(\alph{enumi})}
\begin{enumerate}
\item \label{local-global Witt}
We have an exact sequence 
\[0 \to P_{F_\bullet} \to \W(F) \to \W(F_1) \times \W(F_2) \to
\W(F_0),\]
where the map on the right is given by taking the difference of the
restrictions of the two Witt classes, and where $P_{F_\bullet}$
is the subgroup of $\W(F)$ consisting of classes of locally hyperbolic binary
Pfister forms; i.e.\ forms $\<1,-d\>$ where $d \in F^\times$ becomes a square in each $F_i$.

\item \label{ker of map on Witt}
The group $P_{F_\bullet}$ is naturally isomorphic to the kernel of the local-global map
\[H^1(F, \mu_2) \to \prod_v H^1(F_v,
\mu_2),\]
i.e.\ the group of square classes in $F^\times$ that become the trivial square class in each $F_v^\times$.

\item \label{factorization for SO}
For every quadratic form $q$ over $F$, factorization holds for the diamond of
groups $\SO(q)(F_\bullet)$.

\item \label{abstract local-global for isotropy}
The local-global principle for isotropy holds for quadratic forms over $F$
that are of dimension unequal to two, and for binary forms that do not lie in $P_{F_\bullet}$.  That is, if such a form $q$ becomes isotropic over $F_v$ for each $v \in \mc V_1 \cup \mc V_2$, then $q$ is isotropic.

\item \label{abstract Witt index}
If $q$ is a regular quadratic form over $F$ then 
\(\wi(q) = \min \{\wi(q_v) \ | \ v \in \mc V_1 \cup \mc V_2
\} - \varepsilon\), where $\varepsilon=1$ if 
$q$ represents a non-trivial class in $P_{F_\bullet}$ and otherwise $\varepsilon=0$.
\end{enumerate}
\end{thm}

\begin{proof}
{\em Proof of part~(\ref{local-global Witt})}:
Let 
$P_{F_\bullet}$  be the kernel of the diagonal map $\W(F) \to \W(F_1) \times \W(F_2)$,
and let $\alpha \in P_{F_\bullet} \subseteq
\W(F)$.  Thus $\alpha_{F_v} = 0$ for each $v$.  Here $\alpha$ is the
class of a quadratic form $q$ such that
$q_{F_v}$ is hyperbolic; hence $q$ is of even dimension 
$2n$. Let $d$ be the discriminant of $q$ and let $b = \left<1,
-d\right>$.  Since $q_{F_v}$ has trivial discriminant for
each $v$, it follows that $b_{F_v}$ is hyperbolic for each $v$; i.e.\ $d \in (F_v^\times)^2$.  Hence the form $q' = q \perp -b$ has trivial discriminant and is
hyperbolic for each $v$.  Thus $q'$ corresponds to a class in $H^1(F,
\SO((n+1)h))$. 

Let $H$ be the $\SO((n+1)h)$-torsor corresponding to $q'$. Since $q'_{F_v}$
is hyperbolic for each $v$, we see that $H(F_i) \neq \varnothing$ for $i
= 1, 2$. But by Theorem~\ref{factor local-global}, since factorization holds for
$\SO((n+1)h)(F_\bullet)$, it follows that $H(F)$ is non-empty. Hence the torsor $H$ is
split, and so $q' = q \perp -b$ is hyperbolic. But this implies that
$q$ is equivalent to the binary Pfister form $b$.

For exactness on the right, suppose that we have Witt
classes $\alpha_i \in \W(F_i)$ such that $(\alpha_1)_{F_0} =
(\alpha_2)_{F_0}$. We wish to show that there is a class $\alpha \in
\W(F)$ such that $\alpha_{F_i} = \alpha_i$ for $i = 1, 2$. To begin,
we choose representative forms $q_i$ over $F_i$ with class $\alpha_i$
of the same dimension $n$. Since these forms become Witt equivalent
over $F_0$, and since they have the same dimension, 
they necessarily become isometric over $F_0$. But the category of
quadratic forms of dimension $n$ under isometry is equivalent to the
category of $\O(n\<1\>)$-torsors; so by Theorem~\ref{patching prop thm}, 
there is a quadratic form $q$ over $F$ such that $q_{F_i} \cong q_i$, as
desired.

\medskip
{\em Proof of part~(\ref{ker of map on Witt})}: 
Since elements of $P_{F_\bullet}$ are represented by binary forms, they lie in the fundamental ideal $I(F)$.  So the exact sequence in part~(\ref{local-global Witt}) restricts to an exact sequence \[0 \to P_{F_\bullet} \to I(F) \to I(F_1) \times I(F_2) \to
I(F_0).\]

We claim that the induced map $P_{F_\bullet} \to I(F)/I^2(F)$ is injective. 
To see this, observe that
if $q$ is a quadratic form whose Witt class is in $P_{F_\bullet}
\cap I^2(F)$, then $q$ has trivial discriminant and even
dimension $2n$, and thus corresponds to a class $\alpha \in H^1(F,
\SO(nh))$ with $\alpha_{F_v}$ trivial for each $v$.
But since
factorization holds for $\SO(nh)(F_\bullet)$, it follows from
Theorem~\ref{factor local-global} that $\alpha$ is split and hence $q$ is hyperbolic.  
Consequently, the above map is injective.

Consider the composition
\[P_{F_\bullet} \to I(F) \to I(F)/I^2(F) \iso F^\times/(F^\times)^2 \iso H^1(F,
\mu_2),\]
which is thus also injective.  Its image is contained in the kernel of the map
$H^1(F, \mu_2) \to \prod_v H^1(F_v, \mu_2)$,
by the definition of $P_{F_\bullet}$ and the functoriality the isomorphism 
$I(F)/I^2(F) \iso H^1(F,\mu_2)$.  The reverse containment follows from 
the description of $P_{F_\bullet}$ in part~(\ref{local-global Witt}), together with 
the fact that the map $I(F) \to F^\times/(F^\times)^2$ takes the class of the quadratic form $\<1,-d\>$ to the square class of $d$.  Hence we obtain the asserted isomorphism.

\medskip

{\em Proof of part~(\ref{factorization for SO})}:
Let $n = \dim q$. Suppose that $q'$
is a quadratic form class with $[q'] \in H^1(F, \SO(q))$ such that
$[q']_{F_v}$ is trivial in $H^1(F_v, \SO(q))$ for each $v$. Then $q \perp -q'$ is a quadratic form over $F$ of
even dimension and trivial discriminant that is trivial over each
$F_v$; and hence its Witt class lies in $P_{F_\bullet}$. Since none of the
nontrivial elements of $P_{F_\bullet}$ have trivial
discriminant, it follows that $q \perp -q'$ is hyperbolic, and that
$q$ and $q'$ are isometric (being of the same dimension).  Hence $[q']$ is trivial in $H^1(F,\SO(q))$. 
This shows that the local-global principle holds for
$\SO(q)$-torsors. By Theorem~\ref{factor local-global}, it follows
that factorization holds for $\SO(q)$.

\medskip

{\em Proof of part~(\ref{abstract local-global for isotropy})}:
We prove the contrapositive; i.e.\ if $q$ is anisotropic then so is some $q_{F_v}$.  This is clear if $q$ is binary but not in the kernel $P_{F_\bullet}$ of the local-global map on Witt rings, 
since a binary form is hyperbolic if and only if it is isotropic.  
So now assume that the anisotropic form $q$ is not binary.  The group $\O(q)$ acts on the projective quadric hypersurface $Q$ defined by $q$, and the action of $\O(q)(F_0)$ is transitive on $Q(F_0)$ by 
the Witt Extension Theorem (see the proof of \cite[Theorem~4.2]{HHK}).
Since $\dim(q)>2$, it follows that $Q$ is connected; and hence $\SO(q)(F_0)$ also acts transitively on $Q(F_0)$. 
By part~(\ref{factorization for SO}) above, factorization holds for
$\SO(q)(F_\bullet)$.  Hence Theorem~\ref{factor local-global}
implies that the local-global principle holds for $Q$.  If $q$ is
anisotropic, then $Q(F)$ is empty and thus some $Q(F_v)$ must be
empty; i.e.\ $q_{F_v}$ is
anisotropic.
 
\medskip

{\em Proof of part~(\ref{abstract Witt index})}:
By Witt decomposition, we are reduced to the case that 
$q$ is anisotropic.    
The desired assertion is clear if the class of $q$ is in $P_{F_\bullet}$, so assume otherwise.  Then some $q_{F_v}$ is anisotropic by part~(\ref{abstract local-global for isotropy}).  Hence $\wi(q)=\min\{\wi(q_v)\}=0$, and thus $\varepsilon=0$ as asserted.
\end{proof}

This abstract result can in particular be applied to the concrete situation of Section~\ref{patches}.

We recall the standing hypotheses: $T$ is a complete discrete valuation ring
with fraction field $K$ and residue field $k$; and $\wh X$ is a projective normal $T$-curve with fraction field $F$ and closed fiber $X$.   
For the remainder of this section on quadratic forms, we {\em additionally assume} that $K$ (or equivalently, $F$) has characteristic unequal to two.  

Theorem~\ref{quad form abstract applic} yields local-global principles for quadratic forms in this context:

\begin{example} \label{quad form local-global on patches}
In the situation of Proposition~\ref{big patch ok},~\ref{local patching}, or~\ref{blowup patch}, 
write $L$ for the field $F$, $F_W$, or $F_P$, respectively. Recall that we
assume $\cha(L) \ne 2$.  
Theorem~\ref{quad form abstract applic} applies because those three
propositions say that patching holds for the given diamond, and because 
factorization holds for $\SO(nh)(F_\bullet)$ 
by Corollary~\ref{rational fact cor} (since
$\SO(nh)$ is a rational connected linear algebraic $F$-group by the Cayley parametrization; e.g.\ see \cite[Remark~4.1]{HHK}).    
\end{example}

The local-global principles given in Example~\ref{quad form local-global on patches} can be carried over to the situation of points on the closed fiber.  First we prove a lemma, with notation as above.

\begin{lemma} \label{approx quad form}
Let $X_0$ be an irreducible component of $X$, with generic point $\eta$.  Let $U_0$ be a non-empty affine open subset of $X_0$ that meets no other irreducible component of $X$, and let $q$ be a quadratic form over $F_{U_0}$.  If $q$ becomes isotropic over $F_\eta$ then $q$ is isotropic over $F_U$ for some non-empty affine open subset $U \subseteq U_0$.
\end{lemma}

\begin{proof}
We may assume that $q$ is a diagonal form $\<a_1,\dots,a_n\>$, with each $a_i \in F_{U_0} \subset F_\eta$.  By \cite[Corollary~3.3(a)]{HHK:Weier}, there exist elements $b_i \in F$ and $c_i \in F_{U_0}^\times$ such that $a_i=b_i c_i^2$.  So $q$ is isometric over $F_{U_0}$ to the form $q':=\<b_1,\dots,b_n\>$ that is defined over $F$.  The projective quadric hypersurface $Q$ over $F$ that is defined by $q'$ has an $F_\eta$-point, since $q'$ is isotropic over $F_\eta$.  
Hence by \cite[Proposition~5.8]{HHK:H1}, $Q$ has an $F_U$-point for some non-empty affine open subset $U \subseteq X_0$.  
After shrinking $U$, we may assume that 
$U \subseteq U_0$.  Thus $q'_U = q_U$, and $q'$ is isotropic over $F_U$.  Hence $q$ is isotropic over $F_U$.
\end{proof}

\begin{prop} \label{quad form local-global on points}
Let $\wh X$ be a projective normal curve with closed fiber $X$ over a complete discrete valuation ring $T$ of characteristic not two.  Let 
$S$ equal $X$, or a non-empty connected open proper subset $W \subset X$, or a non-empty 
connected proper subset $V \subset X$ consisting of a union of  
irreducible components of $X$.  
In these three cases, let 
$L$ respectively equal $F$ or $F_W$ or $F_P$, where in the third case
we consider the model $\wh Y$ of $F$ obtained by contracting $V$ and where the point $P \in \wh Y$ is the image of $V$.  Let $q$ be a quadratic form over $L$.
\renewcommand{\theenumi}{\alph{enumi}}
\renewcommand{\labelenumi}{(\alph{enumi})}
\begin{enumerate}
\item \label{iso_points}
If $\dim(q) \ne 2$ and $q_{F_Q}$ is isotropic for each point $Q \in S$, then $q$ is isotropic.

\item \label{witt_index_points}
If $q$ is a regular
quadratic form, 
then $\wi(q)\in \{\min(\wi(q_{F_Q})),
\min(\wi(q_{F_Q}))-1\}$, where the minimum is taken over all
$Q \in S$.
The second case occurs precisely when $q$ is Witt equivalent to an anisotropic binary
Pfister form that becomes isotropic over each $F_Q$. 

\item \label{local-global witt ring points}
The kernel of the map $\pi:\W(L) \to \prod_{Q\in S} \W(F_Q)$
is equal to the kernel $\Sha_S(L, \mu_2)$
of the local-global map 
$H^1(L, \mu_2) \to \prod_{Q\in S} H^1(F_Q,\mu_2)$ in Galois cohomology.
\end{enumerate}
\end{prop}

\begin{proof}
We begin with the observation that if $q_{F_Q}$ is isotropic for each point $Q \in S$, then there is a finite set $\mc P$ of closed points of $S$ such that $q_{F_U}$ is isotropic for each connected component $U$ of $S \smallsetminus \mc P$.  To see this, note that for each irreducible component $S_0$ of $S$, the form $q_{F_\eta}$ is isotropic, where $\eta$ is the generic point of $S_0$.  
Hence $q$ is isotropic over $F_U$ for some non-empty open subset 
$U \subseteq S_0$, by 
Lemma~\ref{approx quad form}.  We may now take $\mc U$ to be the collection of these sets $U$ (one for each irreducible component of $S$), and take
$\mc P$ to be the complement in $S$ of the union of the sets $U \in \mc U$.  This proves the observation.

We now prove part~(\ref{iso_points}).  By the above observation, 
there are sets $\mc P$ and $\mc U$ as above such that $q_{F_\xi}$ is isotropic for each $\xi \in \mc P \cup \mc U$.  Consider Example~\ref{quad form local-global on patches} 
in the situation of Proposition~\ref{big patch ok},~\ref{local patching}, or~\ref{blowup patch}, if $S$ is equal to $X$, $W$, or $V$ respectively.  By Theorem~\ref{quad form abstract applic}(\ref{abstract local-global for isotropy}) in the context of this example, 
$q$ is isotropic over $L$.

For part~(\ref{witt_index_points}), take the Witt decomposition $q = q_a \perp ih$, where $q_a$ is anisotropic, $h$ is a hyperbolic plane, and $i \ge 0$.  The assertion is trivial if $q$ is itself hyperbolic, and so we may assume that $q_a$ is a non-trivial form.  If $q_a$ remains anisotropic over $F_{Q_0}$ for some ${Q_0} \in S$, then $\wi(q)=i=\wi(q_{Q_0})$, and $\wi(q_Q) \ge i$ for all other $Q \in S$.  Thus $\wi(q)=\min(\wi(q_{F_Q}))$ in this case.

The other case is that $q_a$ becomes isotropic over each $F_Q$.  Then by the above claim, $q_a$ becomes isotropic over $F_\xi$ for each $\xi \in \mc P \cup \mc U$ as above.  So by 
Theorem~\ref{quad form abstract applic}(\ref{abstract Witt index}) in the context of 
Example~\ref{quad form local-global on patches}, 
$q_a$ is an anisotropic binary Pfister form that becomes isotropic (or equivalently, hyperbolic) over each $F_\xi$, and hence over each $F_Q$.  Thus $q$ is Witt equivalent to such a form and 
$\wi(q) = i = \wi(q_{F_Q})-1$ for all $Q \in S$.  

For part~(\ref{local-global witt ring points}), 
observe that by Theorem~\ref{quad form abstract applic}(\ref{local-global Witt},\ref{ker of map on Witt}) in the context of Example~\ref{quad form local-global on patches}, 
the kernel of the map $\pi_{\mc P}:\W(L) \to \prod_\xi \W(F_\xi)$ is equal to
the kernel $\Sha_{\mc P}(L, \mu_2)$ of the local-global map $H^1(L, \mu_2) \to \prod_\xi H^1(F_\xi,
\mu_2)$, where $\xi$ ranges over the elements of $\mc P \cup \mc W$ in each product. 
So it suffices  to show that 
$\ker(\pi) = \bigcup \ker(\pi_{\mc P})$ and
$\Sha_S(L, \mu_2) = \bigcup \Sha_{\mc P}(L, \mu_2)$, 
where in each case $\mc P$ ranges over the non-empty sets of closed points of $S$.  For any choice of $\mc P$ (and hence of $\mc W$), 
$F_U \subset F_Q$ for all $Q \in U \in \mc W$.  Thus $\ker(\pi_{\mc P}) \subseteq \ker(\pi)$ and $\Sha_{\mc P}(L, \mu_2)
\subseteq \Sha_S(L, \mu_2)$ for all $\mc P$.  It therefore remains to show that 
$\ker(\pi) \subseteq \bigcup \ker(\pi_{\mc P})$ and
$\Sha_S(L, \mu_2) \subseteq \bigcup \Sha_{\mc P}(L, \mu_2)$.

We begin with the first of these inclusions. By Witt decomposition, every non-trivial class in $\ker(\pi)$ is represented by a non-trivial anisotropic form $q$.  
Such a $q$ becomes hyperbolic and hence isotropic over each $F_Q$.
So by the observation at the beginning of the proof, 
$q$ becomes isotropic over $F_\xi$ for every $\xi \in \mc P \cup \mc W$, for some choice of $\mc P$.  
Also, since $q$ is not hyperbolic over $F$ but becomes hyperbolic over each $F_Q$, 
it follows from part~(\ref{witt_index_points}) above that $q$ is a binary form.  Thus the isotropic forms $q_{F_\xi}$ are also binary and hence hyperbolic, and so the class of $q$ lies in $\ker(\pi_{\mc P})$ as desired.

To prove the second inclusion, note that a non-trivial element of 
$\Sha_S(L, \mu_2)$ is given by a quadratic field extension of $L$, of the form $L[a^{1/2}]$ for some non-square $a \in L^\times$.  
By definition of $\Sha_S(L, \mu_2)$, the element $a$ is a square in $L_\eta =F_\eta$ for every generic point $\eta$ of $S$.  
For each $\eta$, choose an irreducible connected open neighborhood $U_0 \subset S$. 
By Corollary~\ref{Weier cor}(\ref{Weier}), $a=bc^2$ for some 
$b \in F^\times$ and $c \in \wh R_{U_0}^\times$.  Thus $b$ is a square in $F_\eta$; i.e.\ the $\mu_2$-torsor given by $F[b^{1/2}]$ over $F$ has an 
$F_\eta$-point.  By \cite[Proposition~5.8]{HHK:H1}, this torsor has an $F_U$-point for some open neighborhood $U \subseteq U_0$ of $\eta$; i.e.\ $b$ and hence $a$ is a square in $F_U$.  Let $\mc P$ be the complement of the union of the sets $U$ as $\eta$ varies.  Then the given element of $\Sha_S(L, \mu_2)$ lies in $\Sha_{\mc P}(L, \mu_2)$.
\end{proof}

In the above result, the first case (where $S=X$ and $L=F$) was previously shown in Theorem~9.3, Corollary~9.5, and Theorem~9.6 of \cite{HHK:H1}; but here a uniform argument proves all three cases.  
Analogous local-global assertions have been also proven with respect to discrete valuations on $F$ rather than with respect to points on $X$; see \cite[Theorem~3.1]{CPS}
and \cite[Theorem~9.11]{HHK:H1}.
By combining Proposition~\ref{quad form local-global on points} with the strategy used in \cite[Proposition~9.10]{HHK:H1},
we obtain a local-global principle for isotropy over $F_P$:

\begin{prop} \label{loc-glob isotropy at points}
Let $\wh X$ be a normal projective $T$-curve, 
let $P$ be a closed point of $\wh X$, and let $q$ be a quadratic form on $F_P$ of dimension $\ne 2$.  Assume $\cha(k) \ne 2$.  Then $q$ is isotropic over $F_P$ if and only if it is isotropic over $(F_P)_v$ for every discrete valuation $v$ on $F_P$.  

Even more is true: $q$ is isotropic over $F_P$ provided that it is isotropic over 
$(F_P)_v$ for each discrete valuation $v$ on $F_P$ whose restriction to $F$ is induced by a codimension one point on a regular projective model of $F$ over $T$.
\end{prop}

\begin{proof}
The forward direction of the first assertion is trivial.  For the reverse direction, consider a quadratic form $q$ on $F_P$ that is isotropic on $(F_P)_v$ for every discrete valuation $v$ on $F_P$.  We may assume that $q$ is a diagonal form $\<a_1,\dots,a_n\>$, with $a_i \in \wh R_P$.

By resolution of singularities for surfaces (\cite{Abh}, \cite{Lip}) and Weierstrass Preparation (\cite{HHK:Weier}, Corollary~3.7), there is a birational projective morphism $\pi:\wh X' \to \wh X$ such that $\wh X'$ is regular, and such that 
on the pullback $\pi':\wh X'_P \to \Spec(\wh R_P)$ of $\pi$ with respect to $\Spec(\wh R_P) \to \wh X$, 
the support of $q$ is a normal crossing divisor at every point of 
$\pi'^{-1}(P) \subset \wh X'_P$ (which we may identify with 
$V := \pi^{-1}(P) \subset \wh X'$; here the support of $q$ is defined to be the union of the supports of the divisors of the elements $a_i$.)  By the third case of Proposition~\ref{quad form local-global on points}(\ref{iso_points}), in order to show that $q$ is isotropic over $F_P$ it suffices to show that $q$ is isotropic over $F_Q$ for every point $Q \in V$.

First note that by \cite[Proposition~7.5]{HHK:H1}, for any $Q \in V$ and any discrete valuation $v$ on $F_Q$, the restriction of $v$ to $F$ is a (non-trivial) discrete valuation.  Since $F \subseteq F_P \subseteq F_Q$, it follows that the same holds for the restriction of $v$ to $F_P$.

Consider a closed point $Q$ of $V$.  By the condition on the support of $q$ at $Q$, there exists a generating set $\{x,y\}$ for the maximal ideal of $\wh R_Q$ whose support contains that of $q$ in $\Spec(\wh R_Q)$.  Let $v$ be the $y$-adic valuation on $F_Q$, and let $v_0 = v|_{F_P}$.  
Thus $q$ is isotropic over the completion $(F_P)_{v_0}$,
by the previous paragraph and by the hypothesis of this direction of the proposition.  
Hence $q$ is also isotropic over $(F_Q)_v$, which contains $(F_P)_{v_0}$.  Since $(F_Q)_v$ has residue characteristic unequal to two, it follows from
\cite[Lemma~9.9]{HHK:H1} that $q$ is isotopic over $F_Q$. 

The other case is that $Q$ is a generic point of $V$.  Thus $F_Q$ is a complete discretely valued field, say with valuation $v$.   Again, the restriction $v_0$ of $v$ to $F_P$ is a discrete valuation such that $q$ is isotropic over $(F_P)_{v_0}$.  Hence $q$ is also isotropic over $F_Q$, which contains $(F_P)_{v_0}$.  This completes the proof of the reverse implication.

For the last assertion, note that we may assume in the above argument that the model $\wh X'$ has the property that distinct branches of the closed fiber $X'$ at any closed point must lie on 
distinct irreducible components of $X'$.  With respect to this choice of model $\wh X'$, it suffices to check that the valuations used in the above argument are induced by codimension one points on $\wh X'$.

If $Q$ is a generic point of $V$, then this condition is trivial, since $Q$ is itself a codimension one point on $\wh X'$, and the restriction of $v_0$ to $F$ is the valuation associated to that point on that model.  So consider the case that $Q$ is a closed point of $V$, and take the valuation $v_0 = v|_{F_P}$ considered in the argument above.
By the above condition on branches, the hypothesis of \cite[Theorem~3.1(c)]{HHK:Weier} is automatically satisfied; and so there exist $b \in F$ and $c \in \wh R_Q^\times$ such that $y=bc$. 
Here $b  = yc^{-1} \in \wh R_Q$, and $\{x,b\}$ is a generating set for the maximal ideal of $\wh R_Q$.
So there is a unique irreducible component $D$ of the zero locus of $b$ on $\wh X'$ that passes through $Q$.  The $y$-adic valuation $v$ on $F_Q$ is equal to the $b$-adic valuation on $F_Q$, and the valuation $v_0$ on $F_P$ thus restricts to the $b$-adic valuation on $F$.  That is, $v_0|_F$ is the discrete valuation associated to the generic point of $D$, which is of codimension one on $\wh X'$, as desired.
\end{proof}

\begin{lemma} \label{Cohen lemma}
Let $E$ be the fraction field of a two-dimensional Noetherian 
complete local domain $R$.  Then $E$ is isomorphic to a finite separable extension of a field of the form $F_P$.  Moreover if $R$ is regular or equicharacteristic, then $E$ is itself of the form $F_P$.
\end{lemma}

\begin{proof}
If $R$ is regular, then by \cite[Theorem~15]{Cohen}
it is of the form $T[[x]]$ for some complete discrete valuation ring $T$.  Thus
$E=F_P$ with respect to a point on the projective $T$-line.  

In the general case, by \cite[Theorem~16]{Cohen}, $R$ is a finite extension of a two-dimensional regular complete local domain 
having residue field $k$.  So by the previous paragraph,   
$E$ is a finite extension of a field of the form $F_P$.  This extension is automatically 
separable if $\cha(E)=0$; and by \cite[Th\'eor\`eme~7.1]{GabOrg}, 
it can be chosen to be separable if $\cha(E)>0$.  Hence $E$ is 
isomorphic to a finite separable extension of $F_P$.

Finally, if $R$ is equicharacteristic, then by the above it
is a finite generically separable 
extension of some $T[[x]]$,
where $T=k[[t]]$ since it is equicharacteristic.  So $E$ is
a finite separable extension of $k((t,x))$, and thus of the form $F_P$ by 
\cite{HHK:Weier}, Lemma~3.8.
\end{proof}  

Proposition~\ref{loc-glob isotropy at points} and Lemma~\ref{Cohen lemma} then yield the following strengthening of \cite[Theorem~1.2]{HuLGP}):

\begin{cor} \label{LGP isotropy local cor}
Let $E$ be the fraction field of a regular or equicharacteristic two-dimensional complete local ring whose residue field $k$ has characteristic unequal to two.  Then a quadratic form $q$ over $E$ of dimension $\ne 2$ is isotropic if and only if it becomes isotropic over $E_v$ for every discrete valuation $v$ on $E$. 
\end{cor}

\begin{cor} \label{uinv max}
In the situation 
of Proposition~\ref{big patch ok} (resp.~\ref{local patching} or~\ref{blowup patch}), 
with $\cha(K)\ne 2$, 
let $L$ be the field $F$ (resp.~$F_W$ or $F_P$) 
and let $S$ be the set $X$ (resp.~$W$ or $V$). 
\renewcommand{\theenumi}{\alph{enumi}}
\renewcommand{\labelenumi}{(\alph{enumi})}
\begin{enumerate}
\item \label{uinv max ineq}
Then 
$\displaystyle u(L) \le \max_{\xi \in \mc P \cup \mc W} u(F_\xi)
\le \sup_{Q \in S} u(F_Q)$.
\item \label{uinv max equalities}
If the residue field $k$ of $T$ has characteristic unequal to two, then
$\displaystyle u(L) = \max_{\xi \in \mc P \cup \mc W} u(F_\xi)
= \sup_{Q \in S} u(F_Q) = \sup_{v \in \Omega} u(L_v)$, where $\Omega$ is the set of discrete valuations on $L$ whose restriction to $F$ is a discrete valuation that is induced by a codimension one point on a regular model of $F$.
\item \label{uinv max 0,2}
If $k$ is perfect of characteristic two, and $\cha(K) = 0$, then the four quantities
$u(L)$, $\max_{\xi \in \mc P \cup \mc W} u(F_\xi)$,
$\sup_{Q \in S} u(F_Q)$, $\sup_{v \in \Omega} u(L_v)$
are each less than or equal to $8$.
\end{enumerate}
\end{cor}

\begin{proof}
For $\xi \in \mc P \cup \mc W$, the field $F_\xi$ is not quadratically closed, since the integrally closed ring $\wh R_\xi$ is not.  Thus $u(F_\xi) \ge 2$
by \cite{Lam}, Chapter~XI, Example~6.2(1).

If $q$ is a quadratic form over $L$ having dimension greater than 
$\max_{\xi \in \mc P \cup \mc W} u(F_\xi)$, then $q_\xi$ is isotropic for all $\xi \in \mc P \cup \mc W$.  Hence $q$ is isotropic over $L$ by Theorem~\ref{quad form abstract applic}(\ref{abstract local-global for isotropy}) in the context of Example~\ref{quad form local-global on patches}, using that $\dim(q) > 2$.  
This shows that
$u(L) \le \max_{\xi \in \mc P \cup \mc W} u(F_\xi)$.  

Next, if $U \in \mc W$ and $q$ is a quadratic form over $F_U$ of dimension 
greater than $\sup_{Q \in S} u(F_Q)$, then $q_{F_Q}$ is isotropic for every point $Q$ of $U$.  Hence $q$ is isotropic over $F_U$  by Proposition~\ref{quad form local-global on points}(\ref{iso_points}) for $F_U$, using that $\dim(q)>2$.  Thus 
$\max_{\xi \in \mc P \cup \mc W} u(F_\xi)
\le \sup_{Q \in S} u(F_Q)$.  This proves part~(\ref{uinv max ineq}).

Next, we show part~(\ref{uinv max equalities}), assuming that $\cha(k) \ne 2$. 
By part~(\ref{uinv max ineq}), it suffices to show the two inequalities 
$\sup_{Q \in S} u(F_Q) \le \sup_{v \in \Omega} u(L_v) \le u(L)$.

For the first of these inequalities, we may consider just closed points $Q$.  Let $\pi:\wh Y \to \wh X$ be a birational projective morphism such that $\wh Y$ is smooth, and let $\Sigma = \pi^{-1}(S)$.  Applying 
Proposition~\ref{quad form local-global on points}(\ref{iso_points}) to $F_Q$ for every $Q \in S$ at which $\pi$ is not an isomorphism, we see that 
$\sup_{Q \in S} u(F_Q) \le \sup_{Q \in \Sigma} u(F_Q)$.  So after replacing $\wh X$ by $\wh Y$, we may assume that $\wh X$ is regular.  
Next, since $\cha(k)\ne 2$, there is a split cover $\omega:\wh X' \to \wh X$, say with function field extension $F'/F$ and closed fiber $X'$, such that for each $Q' \in S' := \omega^{-1}(S)$ and each $a \in F'_{Q'}$, there exist $b \in F'$ and $c \in F_{Q'}'^\times$ such that $a=bc^2$. (See 
Corollaries~3.3(c) and~3.7 of \cite{HHK:Weier}, the latter applied to the set of non-unibranched points of $S$.) 
By Proposition~5.1 of~\cite{HHK:H1}, the set of isomorphism classes of fields $F'_{Q'}$, for $Q' \in X'$, is the same as the set of isomorphism classes of fields $F_Q$, for $Q \in X$.  
Also, since $\wh X$ is regular, the analogous assertion is true for the fields $F_v$, by Proposition~7.6 of~\cite{HHK:H1}.  
So we may replace $\wh X$ by $\wh X'$, and therefore assume that $\wh X$ satisfies the above factorization condition on elements $a \in F_Q$.  

Now let $q$ be a quadratic form over $F_Q$ for some $Q \in S$, and
assume that $n:=\dim(q) > u(L_v)$ for all $v \in \Omega$.  To prove the first inequality we wish to show that $q$ is isotropic over $F_Q$.  We may assume that $q = \<a_1,\dots,a_n\>$ with $a_i \in F_Q$.  By the above condition, we may write $a_i=b_ic_i^2$ with $b_i \in F$ and $c \in F_Q^\times$.  Replacing $q$ by the $F_Q$-equivalent form $\<b_1,\dots,b_n\>$, we may assume that $q$ is defined over $F$.  
Now let $w$ be any discrete valuation on $F_Q$
whose restriction to $F$ is a discrete valuation induced by 
a codimension one point on a regular projective model of $F$.
Thus $v:=w|_L \in \Omega$, and
$L_v$ is contained in $(F_Q)_w$.
But $q$ is isotropic over $L_v$ by the dimension assumption on $q$, since 
$v\in \Omega$ and 
$q$ is a quadratic form over $L_v$.  Thus $q$ is isotropic over each $(F_Q)_w$.  By 
Proposition~\ref{loc-glob isotropy at points}, it follows that $q$ is isotropic over $F_Q$.  This completes the proof of the first inequality.

For the second inequality, let $v$ be a discrete valuation on $L$.
Since $\cha(k) \ne 2$, the residue field $\kappa_v$ of $L_v$ also has characteristic unequal to two.
So $u(L_v) = 2u(\kappa_v)$ for each $v \in \Omega$ by Springer's theorem \cite[Proposition~2]{Spr}, 
and $2u(\kappa_v) \le u(L)$ by the first part of \cite[Lemma~4.9]{HHK}.  Hence 
$u(L_v) \le u(L)$, concluding the proof of 
part~(\ref{uinv max equalities}).

For part~(\ref{uinv max 0,2}), we first show under the given hypotheses that $u(F_Q) \le 8$ for all $Q \in S$.  Let $q$ be a quadratic form of dimension $9$ over $F_Q$; we wish to show that $q$ is isotropic.  
Since $\cha(L)=\cha(K)=0$, we may assume that $q$ is a diagonal form $\<a_1,\dots,a_9\>$, with $a_i \in \wh R_P$.  We may assume that each $a_i$ is non-zero.  Let $D$ be the union of the supports of the divisors $(a_1),\dots,(a_9),(2)$ on $\Spec(\wh R_Q)$.  
In the special case that $\wh X$ is regular at $Q$ and $D$ has at most a normal crossing at $Q$, \cite[Proposition~4.6]{PS:PIu} asserts that $q$ is isotropic.  
More generally, let $\pi:\wh X' \to \wh X$ be a blow-up such that $\wh X'$ is regular and $D$ has only normal crossings on 
$\Spec(\wh R_Q) \times_{\wh X} \wh X'$, the corresponding blow-up of $\Spec(\wh R_Q)$. 
Then $q$ is isotropic over $F_{Q'}$ for every $Q' \in \pi^{-1}(Q)$, by the special case just shown.  By Proposition~\ref{quad form local-global on points}(\ref{iso_points}), $q$ is isotropic over $F_Q$, as desired.  Thus $u(F_Q) \le 8$.

So by part~(\ref{uinv max ineq}), $\displaystyle u(L) \le \max_{\xi \in \mc P \cup \mc W} u(F_\xi)
\le \sup_{Q \in S} u(F_Q) \le 8$ in this case. 
To complete the proof of part~(\ref{uinv max 0,2}), it suffices to show that $u(L_v) \le 8$ for each $v \in \Omega$.  
If the residue characteristic of $v$ is zero (and thus not two), then $u(L_v) \le u(L)$ by the same argument as at the end of the proof of part~(\ref{uinv max equalities}); and so $u(L_v) \le 8$.  
Otherwise, $v$ is the valuation associated to the generic point of a component of the special fiber of some model of $L$, 
and $u(L_v) \le 8$ as in the first part of the proof of \cite[Theorem~4.7]{PS:PIu}.
\end{proof}

\begin{remark} \label{uinv max rk}
\begin{shortenum}
\renewcommand{\theenumi}{\alph{enumi}}
\renewcommand{\labelenumi}{(\alph{enumi})}
\item \label{dvr lgp quad forms}
Corollary~\ref{uinv max}(\ref{uinv max equalities}) remains valid if one takes the supremum over
{\em all} the discrete valuations on $L$, since the above argument that 
$u(L_v) \le u(L)$ does not use that $v \in \Omega$.  
\item \label{Hu rk}
Corollary~\ref{uinv max}(\ref{uinv max equalities}) is related to 
Theorem~4.9 in \cite{HuLaurent}.  The proof in \cite{HuLaurent} used that the function field there was assumed to be purely transcendental.  Note that $\cha(k) \ne 2$ in that assertion, by a standing hypothesis there.
\item \label{PS dvr rk}
Part~(\ref{uinv max 0,2}) of the corollary extends Theorem~4.7 of \cite{PS:PIu}, which asserted that $u(F) \le 8$ under these hypotheses.
\end{shortenum}
\end{remark}

The next result generalizes an assertion given in \cite[Theorem~4.1]{HHK:Weier} concerning the value of $u(F_U)$ for an open subset $U$ of $X$.  
Here, as in \cite{HHK} and \cite{HHK:Weier}, given a field $E$, $u_s(E)$ denotes the smallest $n$ such that $u(L)\le 2^i n$ for every finitely generated field extension $L/E$ of transcendence degree $i\le 1$.
The proof is as for \cite[Theorem~4.1]{HHK:Weier},
but using Corollary~\ref{Weier cor} in place of the less general \cite[Corollary~3.3(a)]{HHK:Weier}.  As in that result, we need to assume that the residue field $k$ of the complete discrete valuation ring $T$ has characteristic unequal to two (not just that its fraction field $K$ has characteristic $\ne 2$).  

\begin{prop}\label{u-inv of FW}
Let $\wh X$ be a normal projective $T$-curve, and let
$U$ be a non-empty connected open subset of the closed fiber $X$.
Assume that $\cha(k)\ne 2$.
\renewcommand{\theenumi}{\alph{enumi}}
\renewcommand{\labelenumi}{(\alph{enumi})}
\begin{enumerate}
\item \label{FW upper bound for u}
Then $u(F_U) \le 4u_s(k)$.
\item \label{FW lower bound for u}
Let $\til X$ be the normalization of $X$, and let $\til Q \in \til X$ be a closed point lying over some point $Q \in U$.  Then $u(F_U) \ge 4u(\kappa(\til Q))$. 
\end{enumerate}
\end{prop}

See also \cite[Corollary~3.7]{Cuong}.

\begin{thm} \label{combined thm qf}
Let $E$ be one of the following:
\begin{compactenum}
\renewcommand{\theenumi}{(\roman{enumi})}
\renewcommand{\labelenumi}{(\roman{enumi})}
\item \label{combined qf small patch}
the fraction field of a two-dimensional Noetherian complete local domain $R$ that is regular or equicharacteristic; or
\item \label{combined qf big patch}
a finite separable extension of the fraction field of the $t$-adic completion of $T[x]$, where $T$ is a complete discrete valuation ring with uniformizer $t$.
\end{compactenum}
Assume that the residue field $k$ of $R$ (resp.\ $T$) does not have characteristic two.   
\begin{compactenum}
\renewcommand{\theenumi}{\alph{enumi}}
\renewcommand{\labelenumi}{(\alph{enumi})}
\item \label{extension upper bound for u}
Then $u(E) \le 4u_s(k)$.
\item \label{extension value for u}
If $u(k)=u_s(k)$, and if 
$u(k')=u(k)$ for
every finite extension $k'$ of $k$,
then $u(E)=4u(k)$.   
\end{compactenum}
\end{thm}

\begin{proof}
In case~\ref{combined qf small patch}, $E$ is of the form $F_P$ 
by Lemma~\ref{Cohen lemma}.  The assertion then follows in this case from
\cite{HHK:Weier}, Theorem~4.1, by choosing a finite set of points $\mc P$ 
on the closed fiber $X$ of the model $\wh X$, such that $\mc P$ contains $P$ 
and the points where distinct components of $X$ meet.

In case~\ref{combined qf big patch}, $E$ is a finite separable extension of 
$F_U$, where $U = \mbb A^1_k$, viewed as an open subset of the closed fiber of $\mbb P^1_T$.  By Proposition~\ref{descend_exten}(\ref{descend_big_patch}), $E$ is isomorphic to a field $F'_{U'}$ for some finite extension $F'$ of $F$, where $F$ is the fraction field of $T[x]$, and for some non-empty connected open subset $U' \subset X'$. (Here $X'$ is the closed fiber of a projective normal model $\wh X'$ of $F'$.)  By Proposition~\ref{u-inv of FW}, $u(E)=u(F'_{U'})\le 4u_s(k)$, proving part~(\ref{extension upper bound for u}).  For part~(\ref{extension value for u}), $u(E) \le 4u_s(k) = 4u(k)$ by part~(\ref{extension upper bound for u}); and the reverse inequality follows from 
Proposition~\ref{u-inv of FW}(\ref{FW lower bound for u}), using that 
$u(\kappa(\til Q))=u(k)$ by hypothesis, for any closed point $Q \in U'$.
\end{proof}

\begin{example} \label{An2 u-inv extension}
Theorem~\ref{combined thm qf} applies in particular to the broad class of fields $k$ that satisfy Leep's $A_n(2)$ property.   
Recall that $k$ is called an $A_n(2)$ field if for every $r>0$, every system of $r$ homogeneous forms of degree two over $k$ in more than $r\cdot 2^n$ variables has a non-trivial zero 
over a finite extension of $k$ having odd degree.
(see \cite[Section~2]{Leep}).
Every $C_n$ field is an $A_n(2)$ field (\cite{SSS}, Lemma IV.3.7), 
but not conversely.
Although $p$-adic fields are not $C_n$ for any $n$, they are $A_2(2)$ field for all primes $p$, including $p=2$ (see \cite{Leep}, Corollary~2.7).  
Moreover if $k$ is an $A_n(2)$ field then so is every finite extension of $k$ (\cite{Leep}, Theorem~2.5); and the fields $k(z)$ and $k((z))$ are $A_{n+1}(2)$ fields (\cite{Leep}, Theorem~2.3).  Since $u(k)\le 2^n$ for an $A_n(2)$ field
(\cite{Leep}, Theorem~2.2), it follows from the above properties that $u_s(k)\le 2^n$, and thus $u(E) \le 2^{n+2}$ in the notation of the corollary, provided $\cha(k) \ne 2$.
In the special case that $u(k')=2^n$ for every finite extension $k'/k$, it follows from \cite[Lemma~3.2]{Leep} that $u(k)=u_s(k)=2^n$.  Thus the hypotheses of
Theorem~\ref{combined thm qf}(\ref{extension value for u}) hold.
\end{example}

As a consequence, we obtain Theorem~\ref{intro thm qf}:

\begin{proof}[Proof of Theorem~\ref{intro thm qf}]
By Example~\ref{An2 u-inv extension}, the fields $k$ are respectively $A_0(2)$, $A_1(2)$, $A_2(2)$, $A_3(2)$, as are their finite extensions; and moreover the hypotheses of Theorem~\ref{combined thm qf}(\ref{extension value for u}) hold (using also \cite[Theorem~3.4]{Leep} in the last two cases).
We conclude by that theorem.
\end{proof} 

As pointed out to us by David Leep, the case of the theorem for $u\bigl(k((x,t))\bigr)$,
where $k=\mbb Q_p$ or $\mbb Q_p(z)$ or $\mbb Q_p((z))$, can be deduced directly from 
results in \cite{Leep}.  Namely, if $k$ is an $A_n(2)$-field of characteristic unequal to two, then $k((t))(x)$ is an $A_{n+2}(2)$ field by \cite[Theorem~2.3]{Leep} and so 
$u\big(k((t))(x)\big) \le 2^{n+2}$.  If $u(k)=2^n$ (as in the case of $k=\mbb Q_p$  or $\mbb Q_p(z)$ or $\mbb Q_p((z))$), it then follows that $u\big(k((x,t))\big)=2^{n+2}$ by 
\cite[Proposition~5.1]{Leep}.

\smallskip

Theorem~\ref{combined thm qf} also provides 
explicit values for the $u$-invariant in mixed characteristic, when the residue characteristic is odd:

\begin{cor} \label{mixed char uinv}
Let $p$ be an odd prime, and $\mbb Z_p^{\mathrm{ur}}$ the maximal unramified extension of $\mbb Z_p$.  Let $R$ be $\mbb Z_p[[x]]$ or the $p$-adic completion of $\mbb Z_p[x]$
(resp.~the $p$-adic completion of $\mbb Z_p^{\mathrm{ur}}[[x]]$ or of
$\mbb Z_p^{\mathrm{ur}}[x]$).
Let $E$ be the fraction field of a finite extension $S$ of $R$; and if $R=\mbb Z_p[[x]]$ or the $p$-adic completion of $\mbb Z_p^{\mathrm{ur}}[[x]]$, assume that $S$ is regular.
Then $u(E)=8$ (resp.~$4$).
\end{cor}  

\begin{proof}
In the first two cases, let $T = \mbb Z_p$ and apply 
Theorem~\ref{combined thm qf}, using that the hypotheses of
part~(\ref{extension value for u}) hold with $u_s(\mbb F_p)=2$, by \cite[Theorem~4.10]{HHK}.  In the other cases, let $T$ be the $p$-adic completion of $\mbb Z_p^{\mathrm{ur}}$; this is the ring of Witt vectors of $\overline {\mbb F}_p$.  Again the hypotheses of Theorem~\ref{combined thm qf}(\ref{extension value for u}) hold, this time with $u_s(\overline {\mbb F}_p)=1$.
\end{proof}

In the case of mixed characteristic with residue characteristic two, we obtain the following, by combining \cite{PS:PIu} with Theorem~\ref{combined thm qf}:

\begin{cor} \label{char 2 res fld u-inv}
Let $k$ be a complete discretely valued field of characteristic zero whose residue field $\kappa$ is perfect of characteristic two.  
If $E$ is a finite extension of $k((x,t))$ or of the fraction field of $k[x][[t]]$, then $u(E) \le 16$. 
\end{cor} 

\begin{proof} Every finite extension $\lambda$ of $\kappa$ is perfect, so $u(\lambda) \le 2$
by \cite[Corollary~1]{MMW}.  Thus $u(\ell)\le 4$ for every finite extension $\ell$ of $k$,
by a theorem of Springer (\cite[Proposition~2]{Spr}).  By \cite[Theorem~4]{PS:PIu}, $u(F) \le 8$ for every finitely generated extension $F$ of $k$ of transcendence degree one.  Thus $u_s(k)\le 4$.  We conclude by Theorem~\ref{combined thm qf}, where separability holds since $\cha(k)=0$.
\end{proof}

%%%%%%%%%%%%%%%%%
\subsection{Applications to central simple algebras} \label{csa section}
%%%%%%%%%%%%%%%%%

Finally, we turn to applications of our results to central simple algebras, especially concerning the period and index of elements of the Brauer group of fields of the sort considered in Section~\ref{patches}, and their finite extensions.  
In particular, for a finite separable extension $E$ of a field of the form $F_P$ or $F_U$, we find an integer $d$ such that $\ind(\alpha)$ divides $\per(\alpha)^d$ for $\alpha \in \Br(E)$.  
See Theorems~\ref{per-ind combination thm} and~\ref{per-ind comb mixed char}, as well as the associated corollaries, for the precise statements,
which strengthen and extend results in \cite[Section~4]{HHK:Weier}.  
For example, for two-dimensional $p$-adic cases, we obtain a sharp bound for the period-index bound $d$ regardless of the period of $\alpha$; see 
Theorems~\ref{combined thm csa} and~\ref{intro cor csa}.

As in Section~\ref{quad form section}, we begin with an abstract local-global result 
(Theorem~\ref{Brauer}) that applies to diamonds that satisfy patching.  
But unlike the analogous
Theorem~\ref{quad form abstract applic}, Theorem~\ref{Brauer} does not require a 
factorization hypothesis.  This makes it more applicable, permitting its use in conjunction with Corollary~\ref{extension blowup both cases}, which in turn makes possible the period-index applications mentioned above for finite separable extensions of fields $F_P$ and $F_U$. Theorem~\ref{Brauer} also yields local-global results about the value of the period-index bound for the fields $F_P$ and $F_U$; see 
Example~\ref{lgp Brauer} and Corollary~\ref{per-ind max}.

Below we use that if $E$ is a product of fields $F_v$, then 
$\Br(E) = \prod \Br(F_v)$ because an Azumaya algebra over $E$ is the same as a product of central simple $F_v$-algebras.

\begin{thm}\label{Brauer}
Suppose that $F_\bullet=(F\leq F_1,F_2\leq F_0)$ is a diamond of rings with $F$ a field and each $F_i$ a finite direct product of fields.  
Write $F_i = \prod_{v \in \mc V_i} F_v$ where each $F_v$ is a field. Assume moreover that patching holds for $F_\bullet$.
Then:
\begin{enumerate}
\renewcommand{\theenumi}{\alph{enumi}}
\renewcommand{\labelenumi}{(\alph{enumi})}
\item \label{abstract ses for Br}
We have a short exact sequence \(0 \to \Br(F) \to \Br(F_1)
\times \Br(F_2) \to \Br(F_0),\)
where the map on the right is given by taking the difference of the
restrictions of the two Brauer classes.
\item  \label{abstract lgp for ind}
For a class $\alpha \in \Br(F)$, we have
\(\ind(\alpha) = \lcm \{\ind(\alpha_v) \ | \ v \in \mc V_1 \cup \mc
V_2 \}.\)
\end{enumerate}
\end{thm}

\begin{proof}
For any field $L$, the natural
map $\GL_n(L) \to \PGL_n(L)$ is surjective,
by Hilbert's Theorem~90  and the 
long exact cohomology sequence associated to 
the short exact sequence of algebraic groups
\(1 \to \mbb G_m \to \GL_n \to \PGL_n \to 1.\)
Now factorization for $\GL_n$ holds for $F_\bullet$,
since $F_\bullet$ has the patching property (see Theorem~\ref{patching prop thm}(\ref{patching and matrix factorization})).  
The above surjectivity then implies that 
factorization for $\PGL_n$ also holds for $F_\bullet$.
Theorem~\ref{factor local-global} then in turn implies that 
the map of pointed sets \(H^1(F, \PGL_n) \to \prod_v H^1(F_v, \PGL_n)\)
has trivial kernel.

Recall that $H^1(F, \PGL_n)$ classifies isomorphism classes of central simple $F$-algebras of degree $n$ (\cite[p.~396]{BofInv}).
So if $A$ is a central simple
$F$-algebra such that $A_{F_v}$ is split for each $v$, then $A$
itself is split.  Thus the homomorphism
\(\Br(F) \to \prod_v \Br(F_v)\) is injective.

Now, suppose that we have classes
\(\alpha_i \in \Br(F_i) = \prod_{v \in \mc V_i} \Br(F_v)\)
for $i = 1, 2$ such that $(\alpha_1)_{F_0} = (\alpha_2)_{F_0}$. We
wish to show that there is an $\alpha \in \Br(F)$ with $\alpha_{F_i} =
\alpha_i$.
Choose a positive integer $n$ that is divisible by $\ind(\alpha_v)$
for each $v$. We may then choose central simple algebras $A_v$ over
$F_v$ of degree $n$ such that the class of $A_i = \prod_{v \in \mc
V_i} A_v$ in $\Br(F_i) = \prod_{v \in \mc V_i} \Br(F_v)$ is
$\alpha_i$. Since $(\alpha_1)_{F_0} = (\alpha_2)_{F_0}$, the algebras
$(A_1)_{F_0}$ and $(A_2)_{F_0}$ are Brauer equivalent, and thus isomorphic, being of the same degree.  
Using patching for central simple algebras 
(see Example~\ref{other categories}),
there is a central simple $F$-algebra $A$ such that $A_{F_i}
\cong A_i$, compatibly.  This gives exactness of the
given sequence, proving part~(\ref{abstract ses for Br}).

We now turn to part~(\ref{abstract lgp for ind}).
By \cite[Proposition~13.4(iv)]{Pie}, 
$\ind(\alpha_{F_v})$ divides $\ind(\alpha)$. 
It thus suffices to show that if 
each $\ind(\alpha_{F_v})$ divides an integer $i$ then so does $\ind(\alpha)$.
Choose a central simple algebra $A$ with
Brauer class $\alpha$ and with some degree $n>i$. Let $\SB_i$ be
the $i$-th generalized Severi-Brauer variety, parametrizing $ni$-dimensional right
ideals of $A$ (see \cite[p.~334]{vdB}, \cite[Theorem~3.6]{See:BS}, or
\cite[p.~255]{HHK}).
For any field extension $L/F$, the group $\PGL_1(A)(L)$ acts transitively on the $L$-points of 
$\SB_i$, via \cite[Proposition~1.12, Definition~1.9]{BofInv} 
(see also \cite[p.~255]{HHK}).
Also, $\ind(A_L)$ divides $i$ if and only if 
$\SB_i(L) \neq \varnothing$, by 
\cite[Proposition~1.17]{BofInv}.
In particular, $\SB_i(F_1)$ and $\SB_i(F_2)$ are non-empty. 
We claim that factorization holds for
$\PGL_1(A)(F_\bullet)$.  Assuming this for the moment, it follows from Theorem~\ref{factor local-global} that
$\SB_i(F) \neq \varnothing$.  Therefore $\ind(\alpha)$ divides $i$, as
desired.

To complete the proof of part~(\ref{abstract lgp for ind}), it remains to prove the above claim.  By Theorem~\ref{factor local-global},
it suffices to prove that the map 
\(H^1(F, \PGL_1(A)) \to \prod_v
H^1(F_v, \PGL_1(A))\)
has trivial kernel.  
Let $\beta$ be in the kernel of this map.
Now for any field $E$, the cohomology set 
$H^1(E, \PGL_1(A))$ parametrizes the set of isomorphism classes of central simple $E$-algebras of degree $n$, with the trivial element corresponding to the class of $A$
(see~\cite[Proposition~29.1 and p.~396]{BofInv}).  
Let $B$ be a central simple $F$-algebra $B$
of degree $n$ whose isomorphism class corresponds to 
$\beta$.  Thus $B \otimes A^{\op}$ induces the trivial element in $\Br(F_v)$ for all $v$; and hence $B \otimes A^{\op} \in \Br(F)$ is itself trivial
by part~(\ref{abstract ses for Br}).  Equivalently, $\beta$ is 
trivial in $H^1(F, \PGL_1(A))$.  This proves the claim and hence the result.
\end{proof}

\begin{example} \label{lgp Brauer}
Under the hypotheses of Proposition~\ref{big patch ok} (resp.\ Proposition~\ref{local patching} or Proposition~\ref{blowup patch} or Corollary~\ref{extension blowup both cases}), the conclusions of 
Theorem~\ref{Brauer} hold, since its hypotheses hold by those propositions.  In particular, a central simple algebra $A$ over $F$ (resp.\ over $F_W$ or $F_P$ or $E$) is split if and only if it is split over each $F_\xi$ for $\xi \in \mc P \cup \mc W$.  Moreover the index of $A$ is the least common multiple of the indices of the algebras $A_{F_\xi}$.  See also \cite[Theorem~7.2]{HH:FP},  \cite[Theorem~5.1]{HHK}, and
\cite[Theorem~2]{RS}
 for related results.
\end{example}

\begin{cor}  \label{per-ind max}
Let $L$ be the field $F$ (resp.\ $F_W$ or $F_P$ or $E$) in the situation
of Proposition~\ref{big patch ok} (resp.\ Proposition~\ref{local patching} or Proposition~\ref{blowup patch} or Corollary~\ref{extension blowup both cases}).  
Let $d$ be a positive integer and let $\alpha \in \Br(L)$.  
For $\xi \in \mc P \cup \mc W$, let $\alpha_\xi$ be the image of $\alpha$
in $\Br(F_\xi)$.  
If 
$\ind(\alpha_\xi)$ divides $\per(\alpha_\xi)^d$ for all $\xi \in \mc P \cup \mc W$, then $\ind(\alpha)$ divides $\per(\alpha)^d$.
\end{cor}

\begin{proof}
Since 
$\per(\alpha_\xi)$ divides $\per(\alpha)$ for each $\xi \in \mc P \cup \mc W$, we have that
$\ind(\alpha_\xi)$ divides $\per(\alpha)^d$ for each $\xi$.
Since $\ind(\alpha) = \lcm\bigl(\ind(\alpha_\xi)\bigr)_{\xi}$
by Example~\ref{lgp Brauer}, it follows that $\ind(\alpha)$ divides $\per(\alpha)^d$.
\end{proof}

We now turn to results that provide period-index bounds for finite separable extensions of fields of the form $F_P$ and $F_W$, in terms of such bounds for $k$.  Even in the special case of the fields $F_P$ and $F_W$ themselves, the results
strengthen \cite[Corollary~5.10]{HHK} and \cite[Theorem~4.6]{HHK:Weier} by improving the exponent on the period and also considering more general open subsets.
First we prove some lemmas.

\begin{lem} \label{small patch per-ind lemma}
Suppose that $\wh R$ is a excellent complete regular local ring of
dimension $2$ with residue field $k$ and fraction field $F$.
Let $n,d$ be positive integers such
that $\mu_n \subset k$ and such that
$\ind \alpha | (\per \alpha)^d$
for every $n$-torsion Brauer class
$\alpha \in {}_n\!\Br(k)$. Let $\beta \in {}_n\!\Br(F)$ be a Brauer class ramified
only along a regular sequence for $\wh R$. Then $\ind \beta \,|\, (\per
\beta)^{d + 2}$.
\end{lem}

\begin{proof}
Let $\beta \in {}_n\!\Br(F)$ be as above. Thus $m := \per(\beta)$ divides $n$. 
By \cite{Salt:Ramanujan}, Proposition 1.2, 
we may write $\beta = \beta_0 + \gamma_1 +
\gamma_2$, where $\beta_0 \in \Br(\wh R)$, where $\gamma_i$ are the
classes of cyclic algebras of degree $m$ and hence have index dividing
$m$. 
Thus $\gamma_1, \gamma_2$ have periods dividing $m$ and hence the same is true of $\beta_0$.
By \cite[Corollary~IV.2.13]{Milne}, we may identify $\Br(\wh R) =
\Br(k)$ via specialization.
Moreover the index of the class $\beta_0 \in \Br(\wh R) \subset \Br(F)$
divides that of its image in $\Br(k)$, since
specialization induces a natural bijection between
\'etale extensions of $k$ and of $\wh R$, and hence of \'etale splittings of associated central simple algebras.
Thus $\ind(\beta) \,|\, \ind(\beta_0)
\ind(\gamma_1)\ind(\gamma_2) \,|\, \per(\beta_0)^d m^2 \,|\, m^{d+2} =
\per(\beta)^{d + 2}$.
\end{proof}

\begin{lem} \label{big patch per-ind lemma}
Suppose $\wh R$ is a $2$-dimensional excellent ring with fraction
field $F$, $t \in \wh R$,
and $\wh R$ is complete with respect to the $t$-adic topology. Suppose
that $\wh R/t\wh R \cong k[U]$ is the coordinate ring of a regular affine $k$-curve $U$ with $\cha(k) \!\not | \, m$, 
and that $\ind \alpha \,|\, (\per
\alpha)^d$ for all $\alpha \in {}_m\!\Br(k[U])$.  If $\beta \in {}_m\!\Br(F)$ is
ramified only at the support of
$t\wh R$, then $\ind \beta \,|\, (\per \beta)^{d + 1}$.
\end{lem}

\begin{proof}
The ramification of $\beta$ defines a finite connected cover $V \to
U$ of curves of degree $n := \per(\beta)$, together with a generator $\sigma$ of its cyclic Galois group $C_n$. Applying
\cite[Theorem~1.1]{Salt:DivAlg} at each closed point of $U$, and using that
$\beta$ is ramified only at $t\wh R$, it follows that the cover $V \to
U$ is actually unramified and hence \'etale. 
By
\cite[VII.5.5]{SGA4-2}, specialization induces an equivalence of categories
between the \'etale covers of $U$ and of $\Spec(\wh R)$; and hence there is an \'etale algebra $S/\wh R$ that lifts $V \to U$ and
which is Galois with generator $\hat\sigma$ lifting $\sigma$. Let $L$ be the fraction field
of $S$, and consider the cyclic algebra $C = (L/F, \hat\sigma, t)$. By
\cite[Lemma 10.2]{Salt:LN},
$C$ and $\beta$ define the same cyclic cover $V \to U$ and the same Galois generator~$\sigma$.  Thus the Brauer class
$\beta - [C] \in \Br(F)$ is unramified over $\wh R$; i.e.\ it lies in $\Br(\wh R)$, and is represented by an Azumaya algebra $B$ over $\wh R$.  Now $\per([C]) \,|\, \ind([C]) \,|\, n= \per(\beta)$,
and so $\per(B) = \per(\beta - [C]) \,|\, \per(\beta)$.  

By reducing modulo $(t)$, the $\wh R$-algebra $B$ induces an Azumaya algebra $B_0$ over $k[U]$, hence in turn a class in the Brauer group of $k(U)$.  
We claim that $\ind(B)$ divides $i := \ind(B_0)$, over $F$ and $k(U)$ respectively.  To see this, 
consider the $i$-th generalized Severi-Brauer variety $\SB_i$ associated to $B$ over $\wh R$.
Its fiber $(\SB_i)_0$ modulo $(t)$ is the $i$-th generalized Severi-Brauer variety associated 
to $B_0$ over $k[U]$.  Since the index of $B_0$ over $k(U)$ is $i$, there is a $k(U)$-point on 
$(\SB_i)_0$, by the key property of generalized Severi-Brauer varieties (recalled in 
the proof of Theorem~\ref{Brauer} above).  Now $(\SB_i)_0$ is smooth and projective over 
$k[U]$, and $U$ is a regular curve; so the valuative criterion for properness 
implies that this $k(U)$-point extends to a $k[U]$-point, viz.\ a section of 
$(\SB_i)_0 \to U = \Spec(k[U]) = \Spec(\wh R/t\wh R)$.  By Lemma~4.5 of \cite{HHK} and the 
comment after that, this section over $U$ lifts to a section of $\SB_i \to \Spec(\wh R)$.  
The generic point of the image of this section is an $F$-point of $\SB_i$.  Thus $\SB_i(F)$ is 
non-empty, and so the index of $[B] \in \Br(F)$ divides $i$, proving the claim.

Now $\per(B_0) \, | \, \per(B)$, since $\Br(\wh R) \to \Br(k[U])$ is a group homomorphism 
taking $[B]$ to $[B_0]$.  Thus $\per(B_0) \, | \, \per(\beta) = n \, | \, m$.  
The above claim and the hypothesis on $\Br(k[U])$ then yield that
\(\ind(B) \, | \, \ind(B_0) \, | \, (\per\, B_0)^d \, | \, (\per\, B)^d\).
But $\beta = [B] + [C]$.  Hence we have that 
$\ind(\beta) \, | \, \ind(B) \ind(C) \, | \, 
(\per\, B)^d \per(\beta) \, | \, (\per\, \beta)^{d+1}$, as asserted.
\end{proof}

\begin{lem} \label{per-ind reduction lemma}
For a (general) field $L$ and an integer $n$, the following are
equivalent:
\begin{enumerate}
\item \label{reduction lemma per-ind general}
For every finite field extension $L'/L$, and $\alpha \in\, _n\!\Br(L')$, we have
$\ind(\alpha) | \per(\alpha)^d$.
\item  \label{reduction lemma per-ind special}
For every prime $q$ dividing $n$, every finite field extension $L'/L$
and every Brauer class $\alpha \in \Br(L')$ of period $q$, we have $\ind(\alpha) |
q^d$.
\end{enumerate}
Moreover if $\cha(L)$ does not divide $n$ then the same assertion holds with respect to the class of finite {\em separable} extensions $L'/L$ in conditions~\ref{reduction lemma per-ind general} and~\ref{reduction lemma per-ind special}.
\end{lem}

\begin{proof}
The forward implication is trivial.
For (\ref{reduction lemma per-ind special}) $\Rightarrow$ (\ref{reduction lemma per-ind general}),
by considering primary parts we may assume that $\per(\alpha)$ is a prime power $q^r$.  
The implication is then given by \cite[Lemma~1.1]{PS:PIu}.  In the case that 
$\cha(L)$ does not divide $n$, the corresponding assertion for separable extensions $L'/L$ holds because the proof of \cite[Lemma~1.1]{PS:PIu} involves only extensions of $q$-power degree.
\end{proof}

Following \cite{Lieb} and \cite{HHK}, we define the {\em Brauer dimension} of
a field $k$ {\em away from} a prime $p$ as follows: The value is $0$ if the
absolute Galois group of $k$ is a pro-$p$ group (e.g.\ if $k$ is separably
closed).  Otherwise, it is the infimum of the positive integers $d$ such that
for every finite generated field extension $E/k$ of transcendence degree $i\le
1$, every $\alpha \in \Br(E)$ of period prime to $p$ satisfies 
 $\ind\,\alpha \,|\, (\per\,\alpha)^{d+i-1}$.
(We note that the term ``Brauer dimension'' was used in somewhat different
senses in the manuscripts \cite{ABGV} and \cite{PS:PIu}.) 
 
\begin{thm} \label{per-ind combination thm}
Let $T$ be a complete discrete valuation ring whose 
residue field $k$ has Brauer dimension $d$ away from $p := \cha(k)$. 
Let $\wh X$ be a normal projective $T$-curve with function field $F$ and closed fiber $X$.  
Let $\xi$ be either a closed point $P \in X$ or a connected Zariski open subset $W \subset X$, and  
let $E$ be a finite separable extension of $F_\xi$.
Then $\ind(\alpha)\, |\, \per(\alpha)^{d+1}$
for all $\alpha$ in $\Br(E)$ of period not divisible by~$p$.
\end{thm}

\begin{proof}
By hypothesis, $\per(\alpha)$ is not divisible by $\cha(k)$ and hence also not by the characteristic of $K$, the fraction field of $T$.  
By Lemma~\ref{per-ind reduction lemma}, we may assume that $\per(\alpha)$ is a prime number $q \ne \cha(K), \cha(k)$.   Since the degree $[K(\zeta_q):K]$ is prime to $q$, we may also assume that $K,k$ each contain $\mu_q$.  Namely, let $K'=K(\zeta_q)$, $k'=k(\zeta_q)$, and $E':=E(\zeta_q)$.  Then  the period of the induced element
$\alpha' \in \Br(E')$ is equal to the period of $\alpha \in \Br(E)$,
by~\cite[Proposition~14.4.b(v)]{Pie}.  
Also, the index of $\alpha$ (which is a
power of $q$) divides $[E':E]\ind(\alpha')$ by~\cite[Proposition~13.4(v)]{Pie}; and hence it divides the index of $\alpha'$ since $[E':E]$ is relatively prime to $q$.
So we may henceforth assume that $K,k$ each contain $\mu_q$.
  
In the case that $\xi=W$, we may assume that $W$ is affine, since $F_W=F_{W'}$ for some connected affine open set $W'$ on a normal model of $T$ (see Remark~\ref{RW remark}(\ref{contraction remark})).
In both cases $\xi=P,W$, let
$S$ be the integral closure of $\wh R_\xi$ in $E$, let $D$ be the ramification divisor of $\alpha$ on $\Spec(S)$, 
and let $\til \xi$ be the inverse image of $\xi$ under $\Spec(S) 
\to \Spec(\wh R_\xi)$.
Applying Corollary~\ref{extension blowup both cases},
we obtain a birational projective morphism $\pi:\wh V \to \Spec(S)$ and a non-empty set $\mc P$ of closed points of $V := \pi^{-1}(\til \xi)$ satisfying the four conditions there.  

Observe that the assertion holds in the special case that $\pi$ is an isomorphism. 
Namely, if $\xi=P$, 
then 
$\ind\alpha | (\per \alpha)^{d+1}$ for all $\alpha \in {}_q\!\Br(E)$ by Lemma \ref{small patch per-ind lemma}, since $\ind \gamma | (\per \gamma)^{d-1}$ 
for all $\gamma \in {}_q\!\Br(k_{\til P})$ by the assumption on the Brauer dimension of $k$.  
(Here $k_{\til P}$ is the residue field at $\til P$.)
Similarly, if $\xi=W$, then
$\ind\alpha | (\per \alpha)^{d+1}$ for $\alpha \in {}_q\!\Br(E)$ 
by Lemma~\ref{big patch per-ind lemma}.
  
So we may now assume that $\pi$ is not an isomorphism, and therefore that the patching assertion in the last part of Corollary~\ref{extension blowup both cases} holds.  Let $D', \mc W, \mc B$ be as in that result.  
By properties~(\ref{ncd}) and~(\ref{crossing points}) of 
Corollary~\ref{extension blowup both cases}, 
$\alpha_U \in \Br(F_U)$ is unramified away from the closed fiber $U \subset \Spec(\wh R_U)$ for each $U \in \mc W$.

Now $\ind\alpha_Q | (\per \alpha_Q)^{d+1}$ in $\Br(F_Q)$ for each $Q  \in \mc P$ by Lemma \ref{small patch per-ind lemma}, 
and therefore $\ind\alpha_Q | (\per \alpha)^{d+1}$, since $\per \alpha_Q | \per \alpha$.  
Similarly, $\ind\alpha_U | (\per \alpha)^{d+1}$ in $\Br(F_U)$ for all $U  \in \mc W$
by Lemma~\ref{big patch per-ind lemma}.  
But the index of $\alpha \in \Br(E)$ is the least common multiple of the indices 
of all the induced elements $\alpha_Q, \alpha_U$, for $Q \in \mc P$ and $U \in \mc W$, by 
Example~\ref{lgp Brauer} in the context of
Corollary~\ref{extension blowup both cases}.
So the desired conclusion follows.
\end{proof}

Above, we restricted attention to elements of the Brauer group whose period is prime to the residue characteristic $p$.  But in \cite{PS:PIu}, a result was shown about elements whose period is equal to $p$.  Namely, 
suppose that $\cha(K)=0$ and $\cha(k)=p>0$.  If
$F$ is a finitely generated $K$-algebra of transcendence degree one, and if $\alpha \in \Br(F)$ has period $p$, then $\ind(\alpha)$ divides $p^{2n+2}$, where $n$ is the $p$-rank of the residue field $k$ of $T$.  (See \cite[Theorem~3.6]{PS:PIu}.  Recall that the \textit{$p$-rank}, or \textit{imperfect exponent}, of a field $k$ of characteristic $p$ is the integer $n$ such that $[k:k^p]=p^n$.)   
Combining the ideas there with the ideas above, we obtain the following:

\begin{thm} \label{per-ind comb mixed char}
In the situation of Theorem~\ref{per-ind combination thm}, 
suppose that $T$ has mixed characteristic $(0,p)$.  
If the period of $\alpha \in \Br(E)$ is a power of $p$, then $\ind(\alpha)$ divides $(\per\,\alpha)^{2n+2}$, where $n$ is the $p$-rank of $k$.
\end{thm}

\begin{proof}  \textit{Case I}: $\per(\alpha) = p$.

As in the proof of Theorem~\ref{per-ind combination thm}, we may assume that the fraction field $K$ (though not the residue field $k$) contains a primitive $p$-th root of unity.  As in the proof of Theorem~\ref{per-ind combination thm}, we obtain a birational projective morphism $\pi:\wh V \to \Spec(S)$ and an associated non-empty finite set $\mc P \subset V$.  

In the case that $\pi$ is not an isomorphism, we proceed as in the proof of Theorem~\ref{per-ind combination thm} but use \cite[Proposition~3.5 and Theorem~2.4]{PS:PIu} instead of using Lemmas~\ref{small patch per-ind lemma} and~\ref{big patch per-ind lemma}.  Namely, as before we obtain a set $\mc W$ consisting of the components of $V \smallsetminus \mc P$.  If
$U \in \mc W$, then $U$ is irreducible and we may consider its generic point $\eta$.  The $p$-rank of the field $F_\eta$ is $n+1$.  Applying \cite[Theorem~2.4]{PS:PIu} to $F_\eta$, we thus obtain that $\ind(\alpha_{F_\eta})$ divides $p^{2n+2}$.  By \cite[Proposition~5.8]{HHK:H1} and \cite[Proposition~1.17]{BofInv}, after shrinking $U$ (and correspondingly enlarging $\mc P$), we have that $\ind(\alpha_{F_U})=\ind(\alpha_{F_\eta})$, which divides $p^{2n+2}$.  Meanwhile, by \cite[Proposition~3.5]{PS:PIu}, $\ind(\alpha_{F_P})$ divides $p^{2n+2}$ for every $P \in \mc P$.  As in the proof of Theorem~\ref{per-ind combination thm}, we conclude via Example~\ref{lgp Brauer}.

In the case that $\pi$ is an isomorphism, we similarly use those two results in \cite{PS:PIu} instead of Lemmas~\ref{small patch per-ind lemma} and~\ref{big patch per-ind lemma}.  In the case that $\xi=W$, we 
apply \cite[Theorem~2.4]{PS:PIu} at each generic point of $W$, and as above obtain a finite collection of points and open sets that partition $W$.  
We then conclude via Example~\ref{lgp Brauer} as in the above case.

\textit{Case II}: General case.

Let $E'$ be a finite extension of $E$.  Recall that the $p$-rank of $E'$ is also equal to $n$, i.e.\ $[E':E'^p] = [E:E^p]$, because $[E':E] = [E'^p:E^p]$ via the Frobenius isomorphism.  Thus Case~I applies to elements of $\Br(E')$ having period $p$.  The result now follows from Lemma~\ref{per-ind reduction lemma} applied to the field $E$ and the integer $\per(\alpha)$.
\end{proof}

\begin{thm} \label{combined thm csa}
Let $E$ be one of the following:
\begin{compactenum}
\renewcommand{\theenumi}{(\roman{enumi})}
\renewcommand{\labelenumi}{(\roman{enumi})}
\item \label{combined csa small patch}
the fraction field of a two-dimensional Noetherian complete local domain $R$; or
\item \label{combined csa big patch}
a finite separable extension of the fraction field of the $t$-adic completion of $T[x]$, where $T$ is a complete discrete valuation ring with uniformizer $t$.
\end{compactenum}
Assume that the residue field $k$ of $R$ (resp.\ $T$) has 
Brauer dimension $d$ away from $p:=\cha(k)$, and has $p$-rank $n$.  Let $\alpha \in \Br(E)$.
Then $\ind(\alpha)\, |\, \per(\alpha)^{d+1}$ if $p$ does not divide $\per(\alpha)$; and  
$\ind(\alpha)\, |\, \per(\alpha)^{\max(d+1,2n+2)}$ with no restriction on $\per(\alpha)$ if $\cha(E)=0$.
\end{thm} 

\begin{proof}
In case~\ref{combined csa small patch}, Lemma~\ref{Cohen lemma} says that $E$ is a 
finite separable extension of a field of the form $F_P$.
In case~\ref{combined csa big patch}, $E$ is a finite separable extension of $F_U$, where $U =\mbb A^1_k \subset \mbb P^1_T$.  So in both cases, $E$ is of 
the form considered in 
Theorems~\ref{per-ind combination thm} and~\ref{per-ind comb mixed char}.

Those two theorems therefore yield this result if the period of $\alpha$ is either prime to $p$ or a power of $p$.  The general case follows by decomposing $\alpha$ into its primary parts.
\end{proof}

For example, taking $T=k[[t]]$, this theorem applies to finite separable extensions of the fraction field of $k[x][[t]]$, as well as of $k((x,t))$.  This strengthens \cite[Corollary~4.7]{HHK:Weier}.

\begin{cor}
In the situation of Theorem~\ref{combined thm csa}, suppose that $k$ is finite (resp.\ algebraically closed).  If $\cha(E)$ does not divide $\per(\alpha)$
then $\ind(\alpha)$ divides $\per(\alpha)^2$.  Moreover 
$\ind(\alpha)=\per(\alpha)$ in the algebraically closed case if $\cha(k)$ does not divide $\per(\alpha)$.
\end{cor}

\begin{proof}
If $k$ is algebraically closed, then $d=0$ and $n=0$ in the notation of Theorem~\ref{combined thm csa}.  Since the period always divides the index, the assertion in this case follows from the theorem.

If $k$ is finite, then $d=1$ by Wedderburn's Theorem and 
\cite[Theorem~32.19]{Rei}; and $n=0$ since $k$ is perfect.  So again the result follows from the theorem.
\end{proof}

In particular, in the $p$-adic case this yields Corollary~\ref{intro cor csa} and the assertion after it.  See also \cite[Theorem~3.4]{HuDiv} for a related result in the local case.

%%%%%%%%%%%%%%%%%%%%%%%%%%%%%%%%

\noindent {\bf Author Information:}

\medskip

\noindent David Harbater\\
Department of Mathematics, University of Pennsylvania, Philadelphia, PA 19104-6395, USA\\
email: harbater@math.upenn.edu
\medskip

\noindent Julia Hartmann\\
%Lehrstuhl f\"ur Mathematik (Algebra), RWTH Aachen University, 52056 Aachen, Germany;\\
Department of Mathematics, University of Pennsylvania, Philadelphia, PA 19104-6395, USA\\
email: hartmann@math.upenn.edu

\medskip

\noindent Daniel Krashen\\
Department of Mathematics, University of Georgia, Athens, GA 30602, USA\\
email: dkrashen@math.uga.edu

\medskip

\noindent The first author was supported in part by NSF grants DMS-0901164 and DMS-1265290, and NSA grant H98230-14-1-0145. 
The second author was supported by the German Excellence Initiative via RWTH Aachen University and by the German National Science Foundation (DFG).
The third author was supported in part by NSF grant DMS-1007462 and NSF CAREER Grant DMS-1151252.

\end{document}